\documentclass{amsart}

\usepackage{amssymb,verbatim}
\usepackage{esint}
\usepackage[colorlinks=true,urlcolor=blue, citecolor=red,linkcolor=blue,linktocpage,pdfpagelabels, bookmarksnumbered,bookmarksopen]{hyperref}
\usepackage[hyperpageref]{backref}

\newtheorem{lm}{Lemma}[section]
\newtheorem{prop}[lm]{Proposition}
\newtheorem{coro}[lm]{Corollary}
\newtheorem{teo}[lm]{Theorem}
\theoremstyle{definition}
\newtheorem{oss}[lm]{Remark}
\newtheorem{defi}[lm]{Definition}
\newtheorem*{ack}{Acknowledgements}

\numberwithin{equation}{section}

\author[Brasco]{Lorenzo Brasco}
\address[L.\ Brasco]{Dipartimento di Matematica e Informatica
\newline\indent
Universit\`a degli Studi di Ferrara
\newline\indent
Via Machiavelli 35, 44121 Ferrara, Italy}
\email{lorenzo.brasco@unife.it}

\author[Lindgren]{Erik Lindgren}
\address[E. Lindgren]{Department of Mathematics, KTH -- Royal Institute of Technology
\newline\indent
100 44, Stockholm, Sweden}
\email{eriklin@math.kth.se}

\dedicatory{To Peter Lindqvist, a gentleman and $p-$Laplacian master, on the occasion of his 70th birthday}

\title[Uniqueness of extremals]{Uniqueness of extremals\\ for some sharp Poincar\'e-Sobolev constants}

\begin{document}

\begin{abstract} 
We study the sharp constant for the embedding of $W^{1,p}_0(\Omega)$ into $L^q(\Omega)$, in the case $2<p<q$.
We prove that for smooth connected sets, when $q>p$ and $q$ is sufficiently close to $p$, extremal functions attaining the sharp constant are unique, up to a multiplicative constant. This in turn gives the uniqueness of solutions with minimal energy to the Lane-Emden equation, with super-homogeneous right-hand side.
\par
The result is achieved by suitably adapting a linearization argument due to C.-S. Lin. We rely on some fine estimates for solutions of $p-$Laplace--type equations by L. Damascelli and B. Sciunzi.
\end{abstract}

\maketitle
\begin{center}
\begin{minipage}{10cm}
\small
\tableofcontents
\end{minipage}
\end{center}

\section{Introduction} 

\subsection{Setting of the problem}
Let $\Omega\subset\mathbb{R}^N$ be an open set, for $1<p<\infty$ we denote by $\mathcal{D}^{1,p}_0(\Omega)$ the completion of $C^\infty_0(\Omega)$ with respect to the norm
\[
\varphi\mapsto \|\nabla \varphi\|_{L^p(\Omega)}.
\]
It is well-known that, under suitable assumptions on the set $\Omega$, the space $\mathcal{D}^{1,p}_0(\Omega)$ is continuously embedded into $L^q(\Omega)$, provided the exponent $q$ is such that
\[
1\le q\le p^*=\left\{\begin{array}{cc}
\dfrac{N\,p}{N-p},& \mbox{ if } p<N,\\
&\\
\mbox{ any finite exponent}, & \mbox{ if } p=N,\\
&\\
\infty,& \mbox{ if }p> N,
\end{array}
\right.
\]
see for example \cite[Chapter 15, Sections 4 \& 5]{Maz}. In this paper, we are interested in studying the sharp constant for this embedding, i.e., the quantity defined by
\[
\lambda_{p,q}(\Omega)=\inf_{u\in \mathcal{D}^{1,p}_0(\Omega)\setminus\{0\}}\frac{\displaystyle\int_\Omega |\nabla u|^p\, dx}{\displaystyle\left(\int_\Omega |u|^q\,dx\right)^\frac{p}{q}},
\]
when the exponent $q$ is {\it super-homogeneous} and {\it subcritical}, i.e., it satisfies
\[
p<q<p^*.
\] 
This problem can be rewritten in equivalent form as
\begin{equation}
\label{eq:pq}
\lambda_{p,q}(\Omega)=\inf_{u\in \mathcal{D}^{1,p}_0(\Omega)}\left\{\int_\Omega |\nabla u|^p\, dx\, :\, \int_\Omega |u|^q=1\right\}.
\end{equation}
We want to address the question of {\it uniqueness} of extremals for the variational problem \eqref{eq:pq}, provided the latter is well-posed (i.e., it admits a solution).
By uniqueness we mean uniqueness up to the choice of the sign, also referred to as {\it simplicity} of $\lambda_{p,q}(\Omega)$.
\par
We recall that the infimum in \eqref{eq:pq} is positive and it is actually a minimum, whenever $\Omega$ has finite volume. Indeed, in this case the embedding 
\[
\mathcal{D}^{1,p}_0(\Omega)\hookrightarrow L^q(\Omega),
\]
is compact\footnote{We point out that the variational problem \eqref{eq:pq} may be well-posed under more general assumptions on $\Omega$, allowing for suitable classes of unbounded sets, possibly with infinite volume. We do not insist on this point, since in any case our main result will hold for a class of {\it bounded} and {\it smooth} open sets.}, under the above restrictions on $q$. It is also useful to recall that in this case, we have that 
\[
\|\nabla \varphi\|_{L^p(\Omega)}\qquad \mbox{ and }\qquad \|\varphi\|_{W^{1,p}(\Omega)}:=\|\nabla \varphi\|_{L^p(\Omega)}+\|\varphi\|_{L^p(\Omega)},
\]
are equivalent norms on $C^\infty_0(\Omega)$. It is sufficient to observe that on an open set with finite volume we have the Poincar\'e inequality
\[
C_\Omega\,\int_\Omega |\varphi|^p\,dx\le \int_\Omega |\nabla \varphi|^p\,dx,\qquad \mbox{ for every } \varphi\in C^\infty_0(\Omega),
\]
at our disposal.
Thus, under the assumptions we will take on $\Omega$, the space $\mathcal{D}^{1,p}_0(\Omega)$ can be identified with the more common space $W^{1,p}_0(\Omega)$, defined as the closure of $C^\infty_0(\Omega)$ in the usual Sobolev space $W^{1,p}(\Omega)$. In what follows, we will always make this identification.
\par
Finally, it is plain to see that any minimizer $u$  of \eqref{eq:pq} solves the following quasilinear version of the {\it Lane-Emden equation} 
\begin{equation} 
\label{eq:eq}
-\Delta_p u =\lambda\,|u|^{q-2}\,u,\ \mbox{ in }\Omega,
\end{equation}
with $\lambda=\lambda_{p,q}(\Omega)$.

\subsection{Previous results}
The question tackled in this paper is quite classical, it is thus important to recall some existing results, so to put things into the right framework. 
\vskip.2cm\noindent
We start from the case of the Laplacian, i.\,e., we choose $p=2$. We refer the interested reader to \cite{BF} for a more comprehensive overview in this case. 
\par
It is known that for any open bounded connected set $\Omega\subset\mathbb{R}^N$, extremals are unique in the {\it sub-homogeneous regime}, i.e., when $1\le q< 2$, without any regularity assumption on $\Omega$. Actually, in this regime the result is much stronger, since one can infer uniqueness of positive solutions to equation \eqref{eq:eq}. This is a classical result by Brezis and Oswald, see \cite{BO}. By using this result and the fact that any extremal for $\lambda_{2,q}$ must have constant sign (see \cite[Proposition 2.3]{BF}), we get the simplicity property.
\par
In the limit case $q=2$, the quantity $\lambda_{2,2}(\Omega)$ is nothing but the first eigenvalue of the Dirichlet-Laplacian. Thus, we fall into the realm of Linear Spectral Theory, which guarantees again the simplicity property (see for example \cite[Theorem 1.2.5]{He}).
\par
When turning to the {\it super-homogeneous regime} $2^*>q>2$, the picture changes. Uniqueness of extremals is known to hold in balls (see for example \cite[Theorem 2 \& Corollary 1]{KL}) and for planar convex sets (see \cite[Theorem 1]{Lin} and also \cite[Theorem 4.5]{BF}). However, there are examples that show that extremals may not be unique. In  \cite[Proposition 1.2]{Naz} Nazarov shows that simplicity of $\lambda_{2,q}(\Omega)$ fails when $\Omega$ is a sufficiently thin spherical shell. Another counter-example can be found in \cite[Example 4.7]{BF}, where the authors consider a starshaped set consisting of two hypercubes overlapping in a small region near one corner. The example in \cite{BF} is greatly inspired by Dancer's fundamental contributions on multiplicity results for the Lane-Emden equation, see \cite{Da1, Da2}.
\vskip.2cm\noindent
For a general $1<p<\infty$, uniqueness of minimizers (up to multiplicity) holds true again for the {\it sub-homogeneous case} $1\le q<p$. This is a consequence of the stronger uniqueness result for positive solutions of \eqref{eq:eq} contained in \cite[Th\'eor\`eme 1]{DiSa} (see also \cite[Theorem 4]{IO}), which is the quasilinear counterpart of the result by Brezis and Oswald. When combining this result with the fact that extremals for $\lambda_{p,q}$ must have constant sign (see for example the proof of \cite[Theorem 1]{IO} or \cite[Theorem 1.2]{KLP}), we can infer simplicity.
\par
For $q=p$, the quantity $\lambda_{p,p}(\Omega)$ is the first eigenvalue of the $p-$Laplacian with Dirichlet conditions, then it is mandatory to refer to \cite{Lindqvist} for the relevant simplicity result. 
\par
For $p^*>q>p$, the situation is again different. Clearly, this is not a surprise, in light of the case $p=2$ previously discussed. As before, simplicity is known to be true in a ball (see \cite{AY}), but fails in general. The counter-example by Nazarov works in this case, too. We also refer to Kawohl's paper \cite{Kawohl} for the very same example. It is important to recall the precise structure of this counter-example: in \cite{Kawohl, Naz} it is observed that
{\it for every} $q>p$, there exists a sufficiently thin spherical shell such that simplicity for $\lambda_{p,q}$ fails. 
\vskip.2cm\noindent
Finally, the limiting case $q=p^*$ deserves a comment. It is well-known that in this case the variational problem \eqref{eq:pq} is well-defined only for $p>N$. The latter implies that the limit exponent is $q=\infty$ and simplicity is known to hold in bounded {\it convex} sets, as recently proved by Hynd and the second author, see \cite[Theorem 1.1]{HL}.
On the other hand, simplicity can fail already for starshaped sets, see \cite[Section 5]{HL}. It is not difficult to see that the very same counter-examples of \cite{HL} can be adapted to prove more generally that simplicity in starshaped sets fails already for $q$ large enough. Indeed, the counter-examples in \cite{HL} are given by suitable non-convex sets having two orthogonal axis of symmetry, for which one can prove that the extremals for $\lambda_{p,\infty}$ do not inherit the same symmetries. These examples include for instance a thin dumbbell domain or a thin bowtie-type domain, see \cite[Section 5]{HL}.
This lack of symmetry allows to produce at least two distinct extremals, by simply composing an extremal with a reflection. By an easy limit argument, one can show that the same phenomenon must happen for $\lambda_{p,q}$, when is $q$ finite and large enough.

\subsection{Main results}
The main result of this paper asserts that for $q>p$ and $q$ close enough to $p$, simplicity of $\lambda_{p,q}$ must hold, at least in sets which are regular enough. More precisely, we have the following:

\begin{teo}
\label{thm:main}
Let $p>2$ and let $\Omega\subset\mathbb{R}^N$ be an open bounded connected set, with $C^{1,\alpha}$ boundary, for some $0<\alpha<1$. There exists $\overline{q}=\overline{q}(N,p,\Omega)>p$ such that for every $p\le q< \overline{q}$ the extremals of \eqref{eq:pq} are unique, up to the choice of their sign.
\end{teo}
\begin{oss}
In light of the counter-examples discussed above, the previous result is essentially optimal. Indeed, these show that we cannot expect uniqueness to hold for all $q>p$, without any assumptions on the set $\Omega$. It is sufficient to think to the case of the spherical shell. However, we are not able to say whether the $C^{1,\alpha}$ regularity of the boundary is really necessary or not: we will try to explain in the next subsection what are the difficulties in removing this assumption. Finally, we recall that for $p=2$, the very same result of Theorem \ref{thm:main} is true for every open bounded connected set $\Omega$ and even for more general open sets, without any regularity assumption on the boundary, see \cite[Proposition 4.3]{BF}.
\end{oss}
The result of Theorem \ref{thm:main} in turn implies a uniqueness result for the solutions of \eqref{eq:eq} having minimal energy. More precisely, for $\lambda>0$ let us introduce the energy functional
\[
\mathfrak{F}_{q,\lambda}(\varphi)=\frac{1}{p}\,\int_\Omega |\nabla \varphi|^p\,dx-\frac{\lambda}{q}\,\int_\Omega |\varphi|^q\,dx,\qquad \mbox{ for every } \varphi\in W^{1,p}_0(\Omega),
\]
which is naturally associated to \eqref{eq:eq}. In particular, we observe that $u$ is a solution of \eqref{eq:eq} if and only if it is a critical point of $\mathfrak{F}_{q,\lambda}$. Moreover, for a critical point $u$, it is easily seen that
\[
\mathfrak{F}_{q,\lambda}(u)=\left(\frac{\lambda}{p}-\frac{\lambda}{q}\right)\,\int_\Omega |u|^q\,dx.
\]
It is sufficient to test the weak formulation of \eqref{eq:eq} with $u$ itself.
\begin{defi}
With the notation above, we will say that $u\in W^{1,p}_0(\Omega)$ is a {\it nontrivial solution of \eqref{eq:eq} with minimal energy} if 
\[
\int_\Omega |u|^q\,dx=\inf\left\{\int_\Omega |v|^q\,dx\, :\, v\in W^{1,p}_0(\Omega)\setminus\{0\} \mbox{ is a critical point of } \mathfrak{F}_{q,\lambda}\right\}.
\]
\end{defi}
We then obtain the following
\begin{coro}
\label{coro:positive}
Let $\lambda>0$, with the notation and assumptions of Theorem \ref{thm:main}, for every $p< q<\overline{q}$ there exists a unique nontrivial solution of \eqref{eq:eq} with minimal energy, up to the choice of the sign.
\end{coro}

\subsection{Some comments on the proof}
The proof of Theorem \ref{thm:main} is largely inspired by that of \cite[Lemma 3]{Lin}, dealing with the case of the Laplacian, i.e., $p=2$ and $q>2$. The result in \cite{Lin} is actually stronger, as it permits to infer uniqueness of positive solutions to \eqref{eq:eq} for $q$ sufficiently close to $2$ and not only for extremals of $\lambda_{2,q}$. This is achieved under the assumption that $\Omega$ is a smooth\footnote{The smoothness hypothesis in not explicitly stated in \cite[Lemma 3]{Lin}. However, a closer inspection of its proof reveals that this is needed in the argument which uses the moving plane method, at the beginning of page 17 there.} open bounded and {\it convex} set.
\par
The proof of \cite{Lin} proceeds by contradiction and it is based on a linearization argument. In order to clarify the contents of our paper, let us try to sketch the idea of \cite[Lemma 3]{Lin}, by sticking for the moment to the simpler case of extremals for $\lambda_{2,q}$. 
\par
By assuming that simplicity fails for every $q>2$, there must exist a sequence $\{q_n\}_{n\in\mathbb{N}}$ such that $q_n\searrow 2$ and $\lambda_{2,q_{n}}(\Omega)$ admits two distinct positive minimizers $u_n$ and $v_n$. Their difference has the following properties:
\begin{itemize}
\item $u_n-v_n$ is sign-changing on $\Omega$;
\vskip.2cm
\item $u_n-v_n$ solves the {\it linearized} equation
\begin{equation}
\label{linear}
-\Delta \psi=\lambda_{2,q_n}(\Omega)\,V_n\,\psi,\qquad \mbox{ where } V_n:=(q_n-1)\,\int_0^1 ((1-t)\,v_n+t\,u_n)^{q_n-2}\,dt.
\end{equation}
\end{itemize} 
Moreover, by using the minimality, it is not difficult to see that both $u_n$ and $v_n$ must converge to a first positive eigefunction of the Dirichlet-Laplacian, with unit $L^2$ norm. By further using that
\[
\lambda_{2,q_n}(\Omega)\to \lambda_{2,2}(\Omega)\qquad \mbox{ and }\qquad V_n\to 1,
\]
we get that the rescaled difference
\[
\phi_n:=\frac{u_n-v_n}{\|u_n-v_n\|_{L^2(\Omega)}},
\]
converges to a non-trivial limit function $\phi$, which can be proved to be a {\it sign-changing first eigenfunction} of the Dirichlet-Laplacian. This gives the desired contradiction, since first eigenfunctions must have constant sign.
\par
The general case of positive solutions to \eqref{eq:eq} is more complicated, since obtaining that solutions must converge to a first Dirichlet eigenfunction requires some nontrivial a priori ``universal'' estimates, i.e., estimates which are uniform as $q$ converges to $2$. This is tackled by means of an ingenious trick, which exploits the moving plane method: it is only here that the smoothness and convexity assumptions come into play in \cite{Lin}. We also refer to \cite[Theorem 5]{Da2} and \cite[Theorem 4.1]{DGP} for a similar uniqueness result, for some classes of planar sets (symmetric but not necessarily convex).
\vskip.2cm\noindent
The transposition of this method to the case of the $p-$Laplacian is bound to immediately face some huge obstructions, already in the simpler case of extremals for $\lambda_{p,q}$. This is mainly due to the nonlinearity of the $p-$Laplacian: in this case, we can say that the difference $u_n-v_n$ of two positive extremals for $\lambda_{p,q_n}(\Omega)$ solves the linearized equation
\begin{equation}
\label{}
-\mathrm{div}(A_n\,\nabla \psi)=\lambda_{p,q_n}(\Omega)\,V_n\,\psi,
\end{equation}
where $V_n$ is as in \eqref{linear}, but now we have the coefficient matrix $A_n$ which is {\it degenerate elliptic}. More precisely, we have
\[
\frac{1}{C}\, \Big(|\nabla u_n|^{p-2}+|\nabla v_n|^{p-2}\Big)\,|\xi|^2\le \langle A_n\,\xi,\xi\rangle\le C\, \Big(|\nabla u_n|^{p-2}+|\nabla v_n|^{p-2}\Big)\,|\xi|^2,\quad \mbox{ for every }\xi\in\mathbb{R}^N.
\]
Now, inferring some suitable compactness for the rescaled sequence 
\[
\phi_n:=\frac{u_n-v_n}{\|u_n-v_n\|_{L^2(\Omega)}},
\]
and passing to the limit in the linearized equation above is quite problematic. 
It is precisely here that we need global regularity informations on the functions $u_n$ and $v_n$, which in turn call into play for some assumptions on the boundary of $\Omega$. More precisely, we need to know that $|\nabla u_n|$
 and $|\nabla v_n|$ are still integrable, even when raised to some suitable {\it negative} powers.
\par
This in turn permits to obtain a compact embedding in some $L^t$ space, for weighted Sobolev spaces of functions such that
\[
\int_\Omega \left(|\nabla u_n|^{p-2}+|\nabla v_n|^{p-2}\right)\, |\nabla \varphi|^2\,dx<+\infty.
\] 
Here we make use of some striking results proved by Damascelli and Sciunzi in \cite{DS}. Unfortunately, the results in \cite{DS} are stated for positive solutions of a Lane-Emden--type equation
\[
-\Delta_p u=f(u),
\] 
without explicit a priori estimates, thus it is not clear whether these regularity estimates hold uniformly or not, as $n$ goes to $\infty$. In the same way, the embedding results are stated for a positive solution, without making precise in the statement how these embeddings depend on the solution itself. Here as well, we have to guarantee that both the embedding constant and the target space are stable, as $n$ goes to $\infty$.
\par
For these reasons, a non-negligible part of the paper is devoted to reprove these results. We claim no originality here, however it is mandatory to go through the proofs of \cite{DS} and carefully check that all the regularity estimates and the weighted embeddings hold {\it uniformly} for the family of solutions $\{u_n\}_{n\in\mathbb{N}}, \{v_n\}_{n\in\mathbb{N}}$.  
\par
This will show that it is possible to pass to the limit in the linearized equation. Then the conclusion of the proof is similar to that of \cite{Lin} exposed above: we will obtain convergence to a non-trivial limit function $\phi$, which can be shown again to be a {\it sign-changing first eigenfunction} of a certain {\it weighted} linear eigenvalue problem. We will show that this is a contradiction, by means of a suitable weighted Picone--type identity.
\vskip.2cm\noindent
Finally, we humbly admit that we have not been able to extend our result to the more general case of positive solutions of \eqref{eq:eq}. This seems quite a challenging task, which we plan to tackle in the future. 

\subsection{Plan of the paper}
In Section \ref{sec:prel} we prove some preliminary facts, which will be needed for the linearization argument previously discussed. In particular, we devote this section to prove uniform (with respect to $q$) regularity estimates for positive extremals of \eqref{eq:pq}, as well as to discuss the stability of weighted embeddings, with respect to a varying weight. Here we need to go through the proofs of \cite{DS}.
\par
In Section \ref{sec:3}, we analyze the first eigenvalue of a certain weighted linear eigenvalue problem. This is a crucial ingredient for the proof of our main result. We prove in particular the existence of a first eigenfunction and its uniqueness, up to a multiplicative constant.
\par
The central part of the paper is then Section \ref{sec:proof}, where we prove Theorem \ref{thm:main}, along the lines detailed above. This section also contains the proof of Corollary \ref{coro:positive}.
\par
Finally, we include three appendices: Appendix \ref{sec:A} and Appendix \ref{sec:B} both contain some technical facts, while Appendix \ref{sec:C} contains the crucial regularity results by Damascelli and Sciunzi, with a uniform control on the relevant a priori estimates.

\begin{ack}
We wish to thank Giovanni Franzina, who first drew our attention on Lin's paper \cite{Lin}. L.B. wants to thank Giulio Ciraolo for first introducing him to the results by Damascelli and Sciunzi, some years ago. We are grateful to Vladimir Bobkov and Grey Ercole for pointing out the papers  \cite{Ta} and \cite{Er2}, respectively.  
\par 
Part of this work has been done during a visit of E.\,L. to Bologna in October 2018 and a visit of L.\,B. to Stockholm in March 2019. Hosting institutions are gratefully acknowledged. Erik Lindgren was supported by the Swedish Research Council, grant no. 2017-03736. 
\end{ack}

\section{Preliminaries}
\label{sec:prel}

\subsection{Notation}
Here we briefly fix some notations that we are going to use throughout the paper.
We will indicate by $\omega_N$ the measure of the $N-$dimensional open ball, with radius $1$.
\par
If $u\in L^1_{\rm loc}(\mathbb{R}^N)$, we set 
\[
u_+:=\max\{u,0\}\qquad \mbox{ and }\qquad u_-:=\max\{-u,0\}.
\]
For an $N\times N$ matrix $A=(a_{i,j})_{i,j=1}^N$ with real coefficients, we will consider its norm
\[
|A|=\left(\sum_{i,j=1}^N |a_{i,j}|^2\right)^\frac{1}{2}.
\] 

\subsection{Basic properties}

We start with a well-known result, i.e. the fact that minimizers of \eqref{eq:pq} have constant sign.
We give here an elementary proof. Observe that the result holds true regardless of the fact that $\Omega$ is connected or not and no regularity assumptions are needed. For simplicity, we assume $\Omega$ to have finite volume: as explained in the introduction, this is a sufficient condition to have well-posedness of \eqref{eq:pq}.
\begin{lm}
Let $1<p<q<p^*$ and let $\Omega\subset\mathbb{R}^N$ be an open set, with finite volume. Then every solution of \eqref{eq:pq} must have constant sign.
\end{lm}
\begin{proof} Let $u\in W^{1,p}_0(\Omega)$ be a solution of \eqref{eq:pq} and let us suppose that $u_+\not\equiv 0$.
By minimality, we have that $u$ verifies
\[
\int_\Omega \langle |\nabla u|^{p-2}\,\nabla u,\nabla \varphi\rangle\,dx=\lambda_{p,q}(\Omega)\,\int_\Omega |u|^{q-2}\,u\,\varphi\,dx,\qquad \mbox{ for every } \varphi\in W^{1,p}_0(\Omega).
\] 
By choosing the test function $\varphi=u_+$, we get
\[
\int_\Omega |\nabla u_+|^p\,dx=\lambda_{p,q}(\Omega)\,\int_\Omega u_+^q\,dx.
\]
Thanks to the fact that $p/q<1$ and that $u$ has unitary $L^q$ norm, we have 
\begin{equation}
\label{normetta}
\int_\Omega u_+^q\,dx\le \left(\int_\Omega u_+^q\,dx\right)^\frac{p}{q},
\end{equation}
with strict inequality, unless the $L^q$ norm of $u_+$ is $1$. The last two equations imply that we must have
\[
\lambda_{p,q}(\Omega)\le \frac{\displaystyle \int_\Omega |\nabla u_+|^p\,dx}{\displaystyle\left(\int_\Omega |u_+|^q\,dx\right)^\frac{p}{q}}\le \frac{\displaystyle \int_\Omega |\nabla u_+|^p\,dx}{\displaystyle\int_\Omega |u_+|^q\,dx}=\lambda_{p,q}(\Omega).
\]
Thus equality must hold everywhere, in particular equality in \eqref{normetta} gives that $u_+$ has unitary $L^q$ norm. By recalling the normalization taken on $u$, this implies that $u=u_+$. This gives the desired conclusion.
\end{proof}
The following technical lemma holds for positive solutions of \eqref{eq:eq}, not necessarily minimizers of \eqref{eq:pq}. It states that if there are two different solutions of \eqref{eq:eq}, their difference must change sign. Here as well, no regularity assumptions on $\Omega$ are needed.
\begin{lm}
\label{lm:nonempty}
Let $1<p<q<p^*$ and let $\Omega\subset\mathbb{R}^N$ be an open connected set, with finite volume. Let  $\lambda>0$ and let $u,v\in W^{1,p}_0(\Omega)$ be two distinct positive solutions of the Lane-Emden equation \eqref{eq:eq}. Then we must have
\[
|\{x\in\Omega\, :\, u(x)>v(x)\}|>0\qquad \mbox{ and }\qquad |\{x\in\Omega\, :\, u(x)<v(x)\}|>0.
\]
In other words, the difference $u-v$ must change sign in $\Omega$.
\end{lm}
\begin{proof}
We argue by contradiction and suppose, for example, that
\[
|\{x\in\Omega\, :\, u(x)<v(x)\}|=0.
\]
Thus we are assuming that
\begin{equation}
\label{over}
u(x)\ge v(x)\ \mbox{ for a.\,e. }x\in\Omega\qquad \mbox{ and }\qquad |\{x\in\Omega\, :\, u(x)>v(x)\}|>0.
\end{equation}
Let $\{v_n\}_{n\in\mathbb{N}}\subset C^\infty_0(\Omega)$ be a sequence such that
\[
\lim_{n\to\infty} \|\nabla v_n-\nabla v\|_{L^p(\Omega)}=0,
\]
which exists by definition of $W^{1,p}_0(\Omega)$. We can assume each $v_n$ to be non-negative and that we have almost everywhere convergence in $\Omega$, as well.
For every $\varepsilon>0$, we take the admissible test function $\varphi=v_n^p/(u+\varepsilon)^{p-1}$ in the weak formulation of the equation for $u$. This yields
\[
\begin{split}
\lambda\,\int_\Omega u^{q-1}\,\frac{v^p_n}{(u+\varepsilon)^{p-1}}\,dx&=\int_\Omega \left\langle |\nabla u|^{p-2}\,\nabla u,\nabla \left(\frac{v_n^p}{(u+\varepsilon)^{p-1}}\right)\right\rangle\,dx\\
&=\int_\Omega \left\langle |\nabla (u+\varepsilon)|^{p-2}\,\nabla (u+\varepsilon),\nabla \left(\frac{v_n^p}{(u+\varepsilon)^{p-1}}\right)\right\rangle\,dx\\
&\le \int_\Omega |\nabla v_n|^p\,dx.
\end{split}
\]
In the last inequality, we used {\it Picone's inequality} for the $p-$Laplacian, see \cite{AH}. By taking the limit as $n$ goes to $\infty$, we thus get
\[
\lambda\,\int_\Omega u^{q-1}\,\frac{v^p}{(u+\varepsilon)^{p-1}}\,dx\le \int_\Omega |\nabla v|^p\,dx.
\]
On the other hand, by using that $v$ is a solution of the same equation, we get
\[
\int_\Omega |\nabla v|^p\,dx=\lambda\,\int_\Omega v^q\,dx.
\]
We thus obtain for every $\varepsilon>0$
\[
\int_\Omega u^{q-1}\,\frac{v^p}{(u+\varepsilon)^{p-1}}\,dx\le \int_\Omega v^q\,dx.
\]
By taking the limit as $\varepsilon$ goes to $0$, using Fatou's Lemma and the fact that $u$ is positive, from the previous estimate we get
\[
\int_\Omega u^{q-p}\,v^p\le \int_\Omega v^q\,dx \qquad \mbox{ that is }\qquad \int_\Omega (u^{q-p}-v^{q-p})\,v^p\,dx\le 0.
\]
Since $q-p>0$ and $v>0$ in $\Omega$ by the minimum principle and the connectedness assumption, we get a contradiction with \eqref{over}.
\end{proof}
\begin{oss}
We seize the opportunity to notice that, in the semilinear case $p=2$, the previous result can be obtained by using the fact that
\[
\int_\Omega (v\,\Delta u-u\,\Delta v)\,dx=0,
\]
as in the proof of \cite[Lemma 3]{Lin}. This is based on the fact that $-\Delta$ is a self-adjoint operator, thus this proof can not be extended to the case $p\not =2$. We circumvented this difficulty by using a convexity trick, i.e. Picone's inequality.
\end{oss}
\subsection{Uniform estimates}
In the following result, we give an $L^\infty$ estimate for solutions of the Lane-Emden equation. The result is well-known, but here the main focus is on the precise form of the a priori estimate.
\begin{prop}[$L^\infty$ estimate]
\label{prop:Linfty}
Let $1<p< q_0<p^*$ and let $\Omega\subset\mathbb{R}^N$ be an open set, with finite volume. 
For every $p\le q\le q_0$, let $u\in W^{1,p}_0(\Omega)$ be a positive weak solution of the Lane-Emden equation \eqref{eq:eq},
for some $\lambda>0$.
Then we have $u\in L^\infty(\Omega)$, with the following estimate
\[
\|u\|_{L^\infty(\Omega)}\le
C\, \left(\lambda^\frac{N}{p\,q}\,\|u\|_{L^q(\Omega)}\right)^\frac{p\,q}{p\,q-(q-p)\,N},
\]
for a constant $C=C(N,p,q_0)>0$.
\end{prop}
\begin{proof}
As already said, the fact that $u\in L^\infty(\Omega)$ is well-known, we focus on obtaining the precise a priori estimate, through a Moser's iteration. The function $u$ solves
\[
\int_\Omega \langle |\nabla u|^{p-2}\,\nabla u,\nabla \varphi\rangle\,dx=\lambda\,\int_\Omega u^{q-1}\,\varphi,
\]
for every $\varphi\in W^{1,p}_0(\Omega)$. We take $\beta\ge 1$ and insert the test function
\[
\varphi=u^\beta.
\]
This gives 
\[
\frac{\beta\,p^p}{(\beta+p-1)^p}\,\int_\Omega\left|\nabla u^\frac{\beta+p-1}{p}\right|^p\,dx=\lambda\,\int_\Omega u^{q-1+\beta}\le \lambda\,\|u\|_{L^\infty(\Omega)}^{q-p}\,\int_\Omega u^{\beta+p-1}\,dx.
\]
By observing that for every $\beta\ge 1$ we have
\[
\left(\frac{\beta+p-1}{p}\right)^p\,\frac{1}{\beta}\le \left(\frac{\beta+p-1}{p}\right)^{p-1},
\]
we can rewrite the previous estimate as
\[
\int_\Omega\left|\nabla u^\frac{\beta+p-1}{p}\right|^p\,dx\le \left(\frac{\beta+p-1}{p}\right)^{p-1}\,\lambda\,\|u\|_{L^\infty(\Omega)}^{q-p}\,\int_\Omega u^{\beta+p-1}\,dx.
\]
We set for simplicity $\vartheta=(\beta+p-1)/p$, thus the previous inequality is equivalent to
\begin{equation}
\label{moser}
\int_\Omega\left|\nabla u^\vartheta\right|^p\,dx\le \vartheta^{p-1}\,\lambda\,\|u\|_{L^\infty(\Omega)}^{q-p}\,\int_\Omega u^{p\,\vartheta}\,dx.
\end{equation}
We now need to distinguish three cases, depeding on whether $p<N$, $p=N$ or $p>N$.
\vskip.2cm\noindent
{\it Case $1<p<N$}: we recall the {\it Sobolev inequality}  
\[
S_{N,p}\,\left(\int_\Omega |\varphi|^{p^*}\,dx\right)^\frac{p}{p^*}\le \int_\Omega |\nabla \varphi|^p\,dx,\qquad \mbox{ for every } \varphi\in W^{1,p}_0(\Omega),
\]
where $p^*=(N\,p)/(N-p)$.
By using this inequality in the left-hand side of \eqref{moser}, we get
\[
S_{N,p}\,\left(\int_\Omega u^{p^*\vartheta}\,dx\right)^\frac{p}{p^*}\le \vartheta^{p-1}\,\lambda\,\|u\|_{L^\infty(\Omega)}^{q-p}\,\int_\Omega u^{p\,\vartheta}\,dx,
\]
and thus
\begin{equation}
\label{moser2}
\left(\int_\Omega u^{p^*\vartheta}\,dx\right)^\frac{1}{p^*\vartheta}\le (\vartheta^\frac{1}{\vartheta})^\frac{p-1}{p}\,\left(\frac{\lambda\,\|u\|_{L^\infty(\Omega)}^{q-p}}{S_{N,p}}\right)^\frac{1}{p\,\vartheta}\,\left(\int_\Omega u^{p\,\vartheta}\,dx\right)^\frac{1}{p\,\vartheta}.
\end{equation}
We define the sequence 
\[
\vartheta_0=\frac{q}{p},\qquad \vartheta_{i+1}=\frac{p^*}{p}\,\vartheta_i=\left(\frac{N}{N-p}\right)^{i+1}\,\frac{q}{p},
\]
and use \eqref{moser2} with $\vartheta_i$. By iterating infinitely many times and observing that
\[
\sum_{i=0}^\infty \frac{1}{p\,\vartheta_i}=\frac{1}{p}\,\frac{p}{q}\,\sum_{i=0}^\infty \left(\frac{N-p}{N}\right)^{i}=\frac{N}{p\,q},
\]
and
\[
\begin{split}
\lim_{n\to\infty} \prod_{i=0}^n \vartheta_i^\frac{1}{\vartheta_i}&=\exp\left(\sum_{i=0}^\infty \frac{\log \vartheta_i}{\vartheta_i}\right)\\
&= \exp\left(\frac{p}{q}\,\log \frac{q}{p}\,\sum_{i=0}^\infty \left(\frac{N-p}{N}\right)^i+\frac{p}{q}\log\frac{N}{N-p}\,\sum_{i=0}^\infty i\,\left(\frac{N-p}{N}\right)^i \right)\\
&\le \exp\left(\log \frac{q_0}{p}\,\sum_{i=0}^\infty \left(\frac{N-p}{N}\right)^i\right)\,\exp\left(\log\frac{N}{N-p}\,\sum_{i=0}^\infty i\,\left(\frac{N-p}{N}\right)^i \right)=:C_0,
\end{split}
\]
we obtain
\begin{equation}
\label{almost}
\|u\|_{L^\infty(\Omega)}\le C_0^\frac{p-1}{p}\, \left(\frac{\lambda}{S_{N,p}}\right)^\frac{N}{p\,q}\,\|u\|_{L^\infty(\Omega)}^\frac{(q-p)\,N}{p\,q}\,\left(\int_\Omega u^q\,dx\right)^\frac{1}{q},
\end{equation}
where $C_0=C_0(N,p,q_0)>0$. We now observe that
\[
\frac{(q-p)\,N}{p\,q}<1\qquad \Longleftrightarrow\qquad q<p^*,
\]
which holds true. Thus from \eqref{almost} we get
\[
\|u\|_{L^\infty(\Omega)}\le C_0^\frac{(p-1)\,q}{p\,q-(q-p)\,N}\, \left(\frac{\lambda}{S_{N,p}}\right)^\frac{N}{p\,q-(q-p)\,N}\,\|u\|_{L^q(\Omega)}^\frac{p\,q}{p\,q-(q-p)\,N},
\]
as desired.
\vskip.2cm\noindent
{\it Case $p=N$}. In this case, we do not have the Sobolev inequality at our disposal. We can replace it with the following {\it Ladyzhenskaya interpolation inequality} for $N<\gamma<\infty$ (see \cite[Theorem 12.83]{Le})
\[
L_{N,\gamma}\,\left(\int_\Omega |\varphi|^\gamma\,dx\right)^\frac{N}{\gamma}\le \left(\int_\Omega |\nabla \varphi|^N\,dx\right)^\frac{\gamma-N}{\gamma}\,\left(\int_\Omega |\varphi|^N\,dx\right)^\frac{N}{\gamma},
\]
for every $\varphi\in W^{1,N}_0(\Omega)$. We use this inequality with the choice $\gamma=2\,N$ in \eqref{moser}. This gives
\[
L_{N,2\,N}\,\left(\int_\Omega u^{2\,N\,\vartheta}\,dx\right)^\frac{1}{2}\le \left(\vartheta^{N-1}\,\lambda\,\|u\|_{L^\infty(\Omega)}^{q-N}\,\int_\Omega u^{N\,\vartheta}\,dx\right)^\frac{1}{2}\,\left(\int_\Omega u^{N\,\vartheta}\,dx\right)^\frac{1}{2}.
\]
After some algebraic manipulations, we get the following replacement of \eqref{almost}
\begin{equation}
\label{moser3}
\left(\int_\Omega u^{2\,N\,\vartheta}\,dx\right)^\frac{1}{2\,N\,\vartheta}\le \left(\vartheta^\frac{1}{\vartheta}\right)^\frac{N-1}{2\,N}\,\left(\frac{\lambda\,\|u\|_{L^\infty(\Omega)}^{q-N}}{L_{N,2\,N}^2}\right)^\frac{1}{2\,N\,\vartheta}\,\left(\int_\Omega u^{N\,\vartheta}\,dx\right)^\frac{1}{N\,\vartheta}.
\end{equation}
We are now in the same situation as above.
We define this time
\[
\vartheta_0=\frac{q}{N},\qquad \vartheta_{i+1}=2\,\vartheta_i=2^{i+1}\,\frac{q}{N},
\]
and use \eqref{moser3} with $\vartheta_i$.
We observe that 
\[
\sum_{i=0}^\infty \frac{1}{2\,N\,\vartheta_i}=\frac{1}{2\,N}\,\frac{N}{q}\,\sum_{i=0}^\infty 2^{-i}=\frac{1}{q},
\]
and
\[
\begin{split}
\lim_{n\to\infty} \prod_{i=0}^n \vartheta_i^\frac{1}{\vartheta_i}&=\exp\left(\sum_{i=0}^\infty \frac{\log \vartheta_i}{\vartheta_i}\right)\\
&=\exp\left(\frac{N}{q}\log \frac{q}{N}\,\sum_{i=0}^\infty \left(\frac{1}{2}\right)^i+\frac{N}{q}\log2\,\sum_{i=0}^\infty i\,\left(\frac{1}{2}\right)^i \right)\\
&\le \exp\left(\log \frac{q_0}{N}\,\sum_{i=0}^\infty \left(\frac{1}{2}\right)^i\right)\,\exp\left(\log 2\,\sum_{i=0}^\infty i\,\left(\frac{1}{2}\right)^i \right)=:C_1,
\end{split}
\]
thus by iterating the estimate we obtain
\[
\|u\|_{L^\infty(\Omega)}\le C_1^\frac{N-1}{2\,N}\, \left(\frac{\lambda}{L^2_{N,2\,N}}\right)^\frac{1}{q}\,\|u\|_{L^\infty(\Omega)}^{1-\frac{N}{q}}\,\left(\int_\Omega u^q\,dx\right)^\frac{1}{q},
\]
where $C_1=C_1(N,q_0)>0$. With some elementary manipulations, we now get the desired estimate.
\vskip.2cm\noindent
{\it Case $p>N$}. This is the easiest case, it is sufficient to use the {\it Morrey--type interpolation inequality} (see Proposition \ref{prop:morrey} below)
\[
\|\varphi\|_{L^\infty(\Omega)}\le Q_{N,p}\,\left(\int_{\Omega} |\nabla \varphi|^p\,dx\right)^\frac{N}{p\,q-(q-p)\,N}\,\left(\int_\Omega |\varphi|^q\,dx\right)^\frac{p-N}{p\,q-(q-p)\,N},\quad \mbox{ for every } \varphi\in W^{1,p}_0(\Omega).
\]
By further observing that from the equation we have
\[
\int_\Omega |\nabla u|^p\,dx=\lambda\,\int_\Omega |u|^q\,dx,
\]
we get
\[
\|u\|_{L^\infty(\Omega)}\le Q_{N,p}\,\lambda^\frac{N}{p\,q\,-(q-p)\,N}\,\left(\,\int_\Omega |u|^q\,dx\right)^\frac{p}{p\,q-(q-p)\,N}.
\]
This concludes the proof.
\end{proof}
The following uniform $C^1$ estimate for solutions of \eqref{eq:pq} will play a crucial role in the proof of our main result. Here we need to enforce the assumptions on $\Omega$ and to work with minimizers of \eqref{eq:pq}.
\begin{teo}
\label{teo:uniformi}
Let $1<p<q_0<p^*$ and let $\Omega\subset\mathbb{R}^N$ be an open bounded connected set, with boundary of class $C^{1,\alpha}$, for some $0<\alpha<1$. For every $p\le q\le q_0$, let $u_q\in W^{1,p}_0(\Omega)$ be a positive minimizer of \eqref{eq:pq}. 
Then there exist $\chi=\chi(\alpha,N,p,q_0, \Omega)\in (0,1)$, $\delta=\delta(\alpha,N,p,q_0,\Omega)>0$ and $\mu_0=\mu_0(\alpha,N,p,q_0,\Omega)>0$, $\mu_1=\mu_1(\alpha,N,p,q_0,\Omega)>0$ such that:
\begin{itemize}
\item $u_q\in C^{1,\chi}(\overline\Omega)$ with the uniform estimate 
\[
\|u_q\|_{C^{1,\chi}(\overline\Omega)}\le L,
\]
for some $L=L(\alpha,N,p,\Omega,q_0)>0$;
\vskip.2cm
\item by defining $\Omega_\delta=\Big\{x\in\Omega\, :\, \mathrm{dist}(x,\partial\Omega)\le \delta\Big\}$, we have
\[
|\nabla u_q|\ge \mu_0,\qquad \mbox{ in } \Omega_\delta,
\]
and 
\[
u_q\ge \mu_1,\qquad \mbox{ in } \overline{\Omega\setminus \Omega_\delta}.
\]
\end{itemize}
\end{teo}
\begin{proof}
As already observed, each $u_q$ is a solution of the quasilinear equation 
\[
-\Delta_p u_q=\lambda_{p,q}(\Omega)\, u_q^{q-1},\qquad \mbox{ in }\Omega,
\]
with homogeneous Dirichlet boundary conditions and the normalization condition
\[
\int_\Omega |u_q|^q\,dx=1.
\]
By \cite{Er} (see also \cite[Theorem 1]{AFI}), we know that the function $q\mapsto \lambda_{p,q}(\Omega)$ is continuous and positive. Thus there exist two constants $\Lambda_0,\lambda_0>0$ depending only on $N,p,\Omega$ and $q_0$ such that 
\begin{equation}
\label{lowerlambda}
\lambda_0\le \lambda_{p,q}(\Omega)\le \Lambda_0,\qquad \mbox{ for every } q\in[p,q_0].
\end{equation}
By using this fact and the normalization condition, we get from Proposition \ref{prop:Linfty} that there exists a constant $C=C(N,p,q_0,\Omega)>0$ such that
\[
\|u_q\|_{L^\infty}\le C,\qquad \qquad \mbox{ for every } q\in[p,q_0].
\] 
The uniform $C^{1,\chi}(\overline\Omega)$ estimate now follows by applying \cite[Theorem 1]{Lie88}.
\vskip.2cm\noindent
As for the uniform lower bound on the gradient, we observe at first that we can apply a suitable version of the Hopf's Lemma (see \cite[Theorem 1]{MS}) to each $u_q$. This yields
\[
\min_{\partial\Omega}|\nabla u_q|>0,\qquad \mbox{ for every } p\le q\le q_0.
\]
By using that the family $\{|\nabla u_q|\}_{p\le q\le q_0}$ has a uniform $C^{0,\chi}(\partial\Omega)$ estimate, an application of Arzel\`a-Ascoli Theorem gives that there exists a constant $\overline\mu>0$ such that 
\[
\min_{\partial\Omega}|\nabla u_q|\ge \overline\mu,\qquad \mbox{ for every } p\le q\le q_0.
\]
We now choose $\delta_0>0$ sufficiently small, such that each point $x\in\Omega_{\delta_0}$ can be uniquely written as
\[
x=x'-|x'-x|\,\nu_\Omega(x'),\qquad \mbox{ with } x'\in\partial\Omega.
\]
This is possible thanks to the regularity of $\partial\Omega$. Here $\nu_\Omega$ stands for the normal outer versor. We then get for every $p\le q\le q_0$ , every $0<\delta\le \delta_0$ and every $x\in\Omega_{\delta}$
\[
|\nabla u_q(x)|\ge |\nabla u_q(x')|-\Big||\nabla u_q(x')|-|\nabla u_q(x)|\Big|\ge \overline\mu-L\,|x'-x|^\chi\ge (\overline\mu-L\,\delta^\chi).
\]
If we now choose
\[
\delta=\min\left\{\left(\frac{\overline\mu}{2}\,\frac{1}{L}\right)^\frac{1}{\chi},\delta_0\right\},
\]
and set $\mu_0=\overline\mu/2$, we obtain
\[
|\nabla u_q(x)|\ge \mu_0,\qquad \mbox{ for every }x\in\Omega_\delta, \, p\le q\le q_0.
\] 
Finally, the uniform lower bound on $u_q$ in $\overline{\Omega\setminus\Omega_\delta}$ can be proved by observing that
\[
\min_{\overline{\Omega\setminus\Omega_\delta}}u_q>0,\qquad \mbox{ for every } p\le q\le q_0,
\] 
thanks to the minimum principle. As before, since the family $\{u_q\}_{p\le q\le q_0}$ is equi-bounded and equi-Lipschitz, by Arzel\`a-Ascoli Theorem we get the existence of $\mu_1>0$ such that
\[
\min_{\overline{\Omega\setminus\Omega_\delta}}u_q\ge \mu_1,\qquad \mbox{ for every } p\le q\le q_0.
\]
This concludes the proof.
\end{proof}
\begin{oss} We remark that the conclusion of the theorem above also holds for $\delta/2$. The lower bound on the gradient is immediate since $\Omega_{\delta/2}\subset \Omega_\delta$. The lower bound on $u_q$ can be deduced with an identical compactness argument.
\end{oss}
\subsection{Weighted embeddings}
The next result is due to Damascelli and Sciunzi, see \cite{DS}. We are interested in the stability both of the embedding constant and of the embedding exponent, with respect to a varying power $q$. We point out that the exponent $\sigma_0$ below is not optimal, but it will be largely sufficient for our purposes.
\begin{teo}[Uniform weighted Sobolev inequality]
\label{teo:WSI}
Let $2<p<q_0<p^*$ and let $\Omega\subset\mathbb{R}^N$ be an open bounded connected set, with boundary of class $C^{1,\alpha}$, for some $0<\alpha<1$. For every $p\le q\le q_0$, let $u_q\in W^{1,p}_0(\Omega)$ be a positive minimizer of \eqref{eq:pq}.
We define
\begin{equation}
\label{sigma0}
\sigma_0=2\,\left(1-\frac{1}{(2\,p-3)\,N}\right)^{-1},
\end{equation}
then for every $2<\sigma<\sigma_0$, there exists $\mathcal{T}=\mathcal{T}(\alpha,N,p,q_0,\sigma,\Omega)>0$ such that
\begin{equation}
\label{sorbola}
\mathcal{T}\,\left(\int_\Omega |\varphi|^\sigma\,dx\right)^\frac{2}{\sigma}\le \int_{\Omega} |\nabla u_q|^{p-2}\,|\nabla \varphi|^2\,dx,\qquad \mbox{ for every } \varphi\in C^\infty_0(\Omega),\ q\in[p,q_0].
\end{equation}
Moreover, such an inequality holds for every $\varphi \in W^{1,p}_0(\Omega)$, as well. The constant $\mathcal{T}$ goes to $0$ as $\sigma\nearrow \sigma_0$.
\end{teo}
\begin{proof}
The inequality follows from \cite[Theorem 3.1]{DS}, by making the choices (with the notations of \cite{DS})
\[
\rho=|\nabla u_q|^{p-2},\qquad p=2,\qquad t=\frac{2\,p-3}{2\,p-4},\qquad \gamma=N-1-\frac{2\,p-3}{2\,p-4}.
\]
However, since we are particularly interested in keeping track of the dependence of the constant $\mathcal{T}$ on the data, we will briefly repeat the proof of \cite{DS}.

For every $\varphi\in C^\infty_0(\Omega)$, we recall the classical representation formula
\[
\varphi(x)=C\,\int_{\mathbb{R}^N} \left\langle \nabla \varphi(y),\frac{x-y}{|x-y|^N}\right\rangle\,dy,
\]
where $C=C(N)>0$, see for example \cite[Lemma 7.14]{GT}.
This in turn implies that
\[
|\varphi(x)|\le C\,\int_\Omega \frac{|\nabla \varphi(y)|}{|x-y|^{N-1}}\,dy,\qquad \mbox{ for every }x\in\Omega.
\]
We still use the notation
\[
t=\frac{2\,p-3}{2\,p-4},\qquad \gamma=N-1-t=N-1-\frac{2\,p-3}{2\,p-4},
\]
then by using H\"older's inequality with exponents 
\[
2\,t\qquad \mbox{ and }\qquad 2\,t/(2\,t-1),
\] 
we obtain
\[
|\varphi(x)|\le C\,\left(\int_\Omega \frac{1}{|\nabla u_q|^{t\,(p-2)}\,|x-y|^\gamma}\,dy\right)^\frac{1}{2\,t}\,\left(\int_\Omega \left(\frac{|\nabla \varphi(y)|\,|\nabla u_q|^\frac{p-2}{2}}{|x-y|^{N-1-\frac{\gamma}{2\,t}}}\right)^\frac{2\,t}{2\,t-1}\,dy\right)^\frac{2\,t-1}{2\,t}.
\]
Observe that by definition
\[
t\,(p-2)=p-\frac{3}{2}<p-1\qquad \mbox{ and }\qquad \gamma<N-2,
\]
thus we can apply Theorem \ref{teo:sweden} with $r=t\,(p-2)$ and get
\begin{equation}
\label{1}
|\varphi(x)|\le C\,\mathcal{S}^\frac{1}{2\,t}\,\left(\int_\Omega \left(\frac{|\nabla \varphi(y)|\,|\nabla u_q|^\frac{p-2}{2}}{|x-y|^{N-1-\frac{\gamma}{2\,t}}}\right)^\frac{2\,t}{2\,t-1}\,dy\right)^\frac{2\,t-1}{2\,t}.
\end{equation}
For simplicity, we now set
\[
F(y)= \left(|\nabla \varphi(y)|\,|\nabla u_q(y)|^\frac{p-2}{2}\right)^\frac{2\,t}{2\,t-1},
\]
and observe that 
\begin{equation}
\label{magia}
\|F\|_{L^\frac{2\,t-1}{t}(\Omega)}^\frac{2\,t-1}{2\,t}=\left(\int_\Omega |F|^\frac{2\,t-1}{t}\,dy\right)^\frac{1}{2}=\left(\int_\Omega |\nabla \varphi|^2\,|\nabla u_q|^{p-2}\,dy\right)^\frac{1}{2}.
\end{equation}
We notice that in view of the choice of $t$, we have
\[
\frac{2\,t-1}{t}=2-\frac{1}{t}=2\,\left(1-\frac{p-2}{2\,p-3}\right)=\frac{2\,p-2}{2\,p-3}>1.
\]
We also introduce the exponent $0<\Theta<N$ given by
\[
\left(N-1-\frac{\gamma}{2\,t}\right)\,\frac{2\,t}{2\,t-1}=N-\Theta\qquad \mbox{ that is }\qquad \Theta=N-\left(N-1-\frac{\gamma}{2\,t}\right)\,\frac{2\,t}{2\,t-1}.
\]
Thanks to the choices of $\gamma$ and $t$, it is not difficult to see that $\Theta$ is positive. More precisely, observe that this exponent is explicitly given by
\[
\Theta=\frac{t-1}{2\,t-1}=\frac{1}{2\,(p-1)}.
\]
In view of these definitions, we can rewrite \eqref{1} as
\begin{equation}
\label{2}
|\varphi(x)|\le C\,\mathcal{S}^\frac{1}{2\,t}\,\left(\int_\Omega \frac{|F(y)|}{|x-y|^{N-\Theta}}\,dy\right)^\frac{2\,t-1}{2\,t},
\end{equation}
so that one can recognize a suitable Riesz potential on the right-hand side. We then recall the classical potential estimate (see for example \cite[Lemma 7.12]{GT})
\begin{equation}
\label{3}
\left\|\int_\Omega \frac{|F(y)|}{|\cdot-y|^{N-\Theta}}\,dy\right\|_{L^m(\Omega)}\le \left(\frac{1-\delta}{\dfrac{\Theta}{N}-\delta}\right)^{1-\delta}\,\omega_N^\frac{N-\Theta}{N}\,|\Omega|^{\frac{\Theta}{N}-\delta}\,\|F\|_{L^s(\Omega)},
\end{equation}
where
\[
0\le \delta:=\frac{1}{s}-\frac{1}{m}<\frac{\Theta}{N}.
\]
We are now ready to finalize the proof of the weighted Sobolev inequality: we choose $2<\sigma<\sigma_0$, where $\sigma_0$ is given by \eqref{sigma0}. Through some lengthy yet elementary computations, we see that this choice guarantees that we have
\[
\frac{\sigma-2}{\sigma}\,\frac{t}{2\,t-1}<\frac{\Theta}{N}.
\]
We then take the $L^\sigma(\Omega)$ norm in \eqref{2}, so to get
\[
\begin{split}
\|\varphi\|_{L^\sigma(\Omega)}&\le C\,\mathcal{S}^\frac{1}{2\,t}\,\left\|\left(\int_\Omega \frac{|F(y)|}{|\cdot-y|^{N-\Theta}}\,dy\right)^\frac{2\,t-1}{2\,t}\right\|_{L^\sigma(\Omega)}\\
&=C\,\mathcal{S}^\frac{1}{2\,t}\,\left\|\int_\Omega \frac{|F(y)|}{|\cdot-y|^{N-\Theta}}\,dy\right\|^\frac{2\,t-1}{2\,t}_{L^{\sigma\,\frac{2\,t-1}{2\,t}}(\Omega)}.
\end{split}
\]
The last term can be estimated from above by using \eqref{3} with the choices
\[
m=\sigma\,\frac{2\,t-1}{2\,t}\qquad \mbox{ and }\qquad s=\frac{2\,t-1}{t}.
\]
These are feasible, since 
\[
\delta=\frac{1}{s}-\frac{1}{m}=\frac{t}{2\,t-1}-\frac{2\,t}{\sigma\,(2\,t-1)}=\frac{\sigma-2}{\sigma}\,\frac{t}{2\,t-1},
\]
is positive and smaller than $\Theta/N$, thanks to the choice of $\sigma$. We then obtain
\[
\|\varphi\|_{L^\sigma(\Omega)}\le C\,\mathcal{S}^\frac{1}{2\,t}\,\left[\left(\frac{1-\delta}{\dfrac{\Theta}{N}-\delta}\right)^{1-\delta}\,\omega_N^\frac{N-\Theta}{N}\,|\Omega|^{\frac{\Theta}{N}-\delta}\right]^\frac{2\,t-1}{2\,t}\,\|F\|^\frac{2\,t-1}{2\,t}_{L^\frac{2\,t-1}{t}(\Omega)}.
\]
By recalling \eqref{magia}, we thus obtained
\[
 \sqrt{\mathcal{T}}\,\|\varphi\|_{L^\sigma(\Omega)}\le \left(\int_\Omega |\nabla u_q|^{p-2}\,|\nabla \varphi|^2\,dx\right)^\frac{1}{2},
\]
with the constant $\mathcal{T}$ given by
\[
\mathcal{T}=\frac{1}{C^2\,\mathcal{S}^\frac{1}{t}}\,\left[\left(\frac{1-\delta}{\dfrac{\Theta}{N}-\delta}\right)^{1-\delta}\,\omega_N^\frac{N-\Theta}{N}\,|\Omega|^{\frac{\Theta}{N}-\delta}\right]^\frac{1-2\,t}{t}.
\]
By recalling that $C=C(N)>0$, that $\mathcal{S}=\mathcal{S}(\alpha,N,p,q_0,\Omega,r,\gamma)>0$, that $r=t\,(p-2)$, that $\Theta$ and $t$ depend on $p$ only, that $\gamma$ depends on $N$  and $p$ and that (finally!) $\delta$ only depends on $\sigma$ (which is fixed) and on $t$ (which depends only on $p$, as already said), we get the desired claim about the quality of the constant $\mathcal{T}$. We further observe that 
\[
\sigma\to \sigma_0 \qquad \Longrightarrow \qquad \delta \to \frac{\Theta}{N},
\]
which shows that $\mathcal{T}$ goes to $0$, as $\sigma$ approaches $\sigma_0$.
\vskip.2cm\noindent
Finally, we prove the last statement. Let us take $\varphi\in W^{1,p}_0(\Omega)$. By definition, there exists a sequence $\{\varphi_n\}_{n\in\mathbb{N}}\subset C^\infty_0(\Omega)$ such that
\[
\lim_{n\to\infty} \Big[\|\varphi_n-\varphi\|_{L^p(\Omega)}+\|\nabla \varphi_n-\nabla \varphi\|_{L^p(\Omega)}\Big]=0.
\]
We first observe that 
\[
\int_\Omega |\nabla u|^{p-2}\,|\nabla \varphi|^2\,dx\le \left(\int_\Omega |\nabla u|^p\,dx\right)^\frac{p-2}{p}\,\|\nabla \varphi\|_{L^p(\Omega)}^2<+\infty,
\]
thanks to H\"older's inequality. We then have
\[
\begin{split}
\Bigg|\int_\Omega |\nabla u|^{p-2}\,|\nabla \varphi_n|^2\,dx&-\int_\Omega |\nabla u|^{p-2}\,|\nabla \varphi|^2\,dx\Bigg|\\
&\le \left(\int_\Omega |\nabla u|^p\,dx\right)^\frac{p-2}{p}\,\left(\int_\Omega \Big||\nabla \varphi_n|^2-|\nabla \varphi|^2\Big|^\frac{p}{2}\,dx\right)^\frac{2}{p}\\
&=\left(\int_\Omega |\nabla u|^p\,dx\right)^\frac{p-2}{p}\,\left(\int_\Omega \Big||\nabla \varphi_n|-|\nabla \varphi|\Big|^\frac{p}{2}\,\Big||\nabla \varphi_n|+|\nabla \varphi|\Big|^\frac{p}{2}\,dx\right)^\frac{2}{p}\\
&\le \left(\int_\Omega |\nabla u|^p\,dx\right)^\frac{p-2}{p}\,\left(\int_\Omega \Big||\nabla \varphi_n|+|\nabla \varphi|\Big|^p\,dx\right)^\frac{1}{p}\\
&\times\left(\int_\Omega \Big||\nabla \varphi_n|-|\nabla \varphi|\Big|^p\,dx\right)^\frac{1}{p}.
\end{split}
\]
This shows that 
\[
\lim_{n\to\infty}\int_\Omega |\nabla u|^{p-2}\,|\nabla \varphi_n|^2\,dx=\int_\Omega |\nabla u|^{p-2}\,|\nabla \varphi|^2\,dx.
\]
On the other hand, by using that $\varphi_n$ converges to $\varphi$ almost everywhere (up to a subsequence), we can apply Fatou's Lemma and get
\[
\liminf_{n\to\infty} \int_\Omega |\varphi_n|^\sigma\,dx\ge \int_\Omega |\varphi|^\sigma\,dx.
\]
By using \eqref{sorbola} for $\varphi_n$, taking the limit as $n$ goes to $\infty$ and using the last two equations in display, we get that \eqref{sorbola} holds for $\varphi$, as well.
\end{proof}

\begin{coro}
\label{coro:nakata}
Let $2<p<q_0<p^*$ and let $\Omega\subset\mathbb{R}^N$ be an open bounded connected set, with boundary of class $C^{1,\alpha}$, for some $0<\alpha<1$. For every $p\le q\le q_0$, let $u_q\in W^{1,p}_0(\Omega)$ be a positive minimizer of \eqref{eq:pq}.
Then there exists an exponent $\theta=\theta(p)\in(1,2)$ and a constant $C=(N,p,q_0,\alpha,\Omega)>0$ such that
\[
\|\varphi\|_{W^{1,\theta}(\Omega)}\le C\, \left(\int_\Omega |\nabla u_q|^{p-2}\, |\nabla \varphi|^2 \,dx\right)^\frac{1}{2},\qquad \mbox{ for every }\varphi\in W^{1,p}_0(\Omega), \, q\in[p,q_0].
\] 
\end{coro}
\begin{proof}
We take $1<\theta <(2\,p-2)/(2\,p-3)$. Then H\"older's inequality with conjugate exponents 
\[
2/\theta\qquad \mbox{ and }\qquad 2/(2-\theta), 
\]
implies
\[
\begin{split}
\left(\int_\Omega |\nabla \varphi|^\theta\, dx\right)^\frac{1}{\theta}
&\leq \left(\int_\Omega |\nabla u_q|^{p-2}\,|\nabla \varphi|^2\, dx\right)^\frac{1}{2}\,\left(\int_\Omega \frac{1}{|\nabla u_q|^{\frac{\theta}{2-\theta}(p-2)}}\, dx\right)^\frac{2-\theta}{2\,\theta}\\
&\leq \widetilde{\mathcal{S}}^\frac{2-\theta}{2\theta}\,\left(\int_\Omega |\nabla u_q|^{p-2}\,|\nabla \varphi|^2\, dx\right)^\frac{1}{2},
\end{split}
\]
where we have used that 
$$
\frac{\theta}{2-\theta}\,(p-2)<(p-1),
$$
which allows us to use estimate \eqref{DS} from Theorem \ref{teo:sweden}, with $r=(p-2)\,\theta/(2-\theta)$. Since $\Omega$ is bounded, by H\"older's inequality we have that $W^{1,p}_0(\Omega)\subset W^{1,\theta}_0(\Omega)$, with continuous inclusion. Moreover, by Poincar\'e inequality 
\[
\|\nabla\varphi\|_{L^\theta(\Omega)}\qquad \mbox{ and }\qquad \|\varphi\|_{W^{1,\theta}(\Omega)},
\]
are equivalent norms on $W^{1,\theta}_0(\Omega)$. These facts conclude the proof.
\end{proof}

A consequence of Theorem \ref{teo:WSI} and Corollary \ref{coro:nakata} is the following compactness result with ``varying weights''. This naturally comes into play when linearizing the equation \eqref{eq:eq}, as explained in the Introduction.

\begin{coro}[Uniform compact embedding]
\label{lem:compact}
Let $2<p<q_0<p^*$ and let $\Omega\subset\mathbb{R}^N$ be an open bounded connected set, with boundary of class $C^{1,\alpha}$, for some $0<\alpha<1$. We take a sequence $\{q_n\}_{n\in\mathbb{N}}\subset [p,q_0]$ and consider accordingly $u_n\in W^{1,p}_0(\Omega)$ a positive minimizer of \eqref{eq:pq} with $q=q_n$. 
If $\{\phi_n\}_{n\in\mathbb{N}}\subset W_0^{1,p}(\Omega) $ is a sequence of functions satisfying
$$
\int_\Omega |\nabla u_n|^{p-2}\, |\nabla \phi_n|^2 \,dx \leq C, \qquad \mbox{ for every } n\in\mathbb{N},
$$
then
$\{\phi_n\}_{n\in\mathbb{N}}$ converges strongly in $L^2(\Omega)$ and weakly in $W^{1,\theta}_0(\Omega)$, up to a subsequence. Here $\theta$ is the same exponent as in Corollary \ref{coro:nakata}.
\end{coro}
\begin{proof} 
The assumption, in conjunction with Corollary \ref{coro:nakata}, entails that $\{\phi_n\}_{n\in\mathbb{N}}$ is a bounded sequence in $W^{1,\theta}_0(\Omega)$. By the classical Rellich-Kondra\v{s}ov Theorem, we get that this sequence converges weakly in $W^{1,\theta}(\Omega)$ and strongly in $L^\theta(\Omega)$, up to a subsequence. Moreover, since $W^{1,\theta}_0(\Omega)$ is also weakly closed, we get that the limit still belongs to $W^{1,\theta}_0(\Omega)$.
\par
In order to get the strong $L^2$ convergence, we observe that, if we denote by $\sigma_0$ the exponent of Theorem \ref{teo:WSI}, for every $2<\sigma<\sigma_0$ and $n,m\in\mathbb{N}$ we have
\[
\begin{split}
\|\phi_n-\phi_m\|_{L^2(\Omega)}&\le \|\phi_n-\phi_m\|_{L^\sigma(\Omega)}^{1-\tau}\,\|\phi_n-\phi_m\|_{L^\theta(\Omega)}^\tau\\
&\le \Big(\|\phi_n\|_{L^\sigma(\Omega)}+\|\phi_m\|_{L^\sigma(\Omega)}\Big)^{1-\tau}\,\|\phi_n-\phi_m\|_{L^\theta(\Omega)}^\tau\\
&\le \left(\frac{1}{\mathcal{T}}\right)^\frac{1-\tau}{2}\,\left(\left(\int_{\Omega} |\nabla u_n|^{p-2}\,|\nabla \phi_n|^2\,dx\right)^\frac{1}{2}+\left(\int_{\Omega} |\nabla u_m|^{p-2}\,|\nabla \phi_m|^2\,dx\right)^\frac{1}{2}\right)^{1-\tau}\\
&\times \|\phi_n-\phi_m\|_{L^\theta(\Omega)}^\tau\\
&\le \left(\frac{4\,C}{\mathcal{T}}\right)^\frac{1-\tau}{2}\,\|\phi_n-\phi_m\|_{L^\theta(\Omega)}^\tau.
\end{split}
\]
We used interpolation in Lebesgue spaces and the uniform weighted Sobolev inequality of Theorem \ref{teo:WSI}, applied to $\phi_n,\phi_m\in W^{1,p}_0(\Omega)$. The above estimate and the strong convergence in $L^\theta(\Omega)$ show that $\{\phi_n\}_{n\in\mathbb{N}}$ is a Cauchy sequence in $L^2(\Omega)$ and thus it strongly converges. By uniqueness of the limit, we conclude.
\end{proof}

\section{A weighted linear eigenvalue problem}
\label{sec:3}

In this section, we treat a {\it weighted} linear eigenvalue problem, that naturally arises when linearizing the quasilinear equation \eqref{eq:eq}. This is decisive in the proof of our main result. 
\par
It is convenient to introduce the notation
\[
\mathcal{H}(z)=\frac{1}{p}\,|z|^p,\qquad \mbox{ for every } z\in\mathbb{R}^N.
\]
Then we observe that 
\[
\nabla \mathcal{H}(z)=|z|^{p-2}\,z,\qquad \mbox{ for every } z\in\mathbb{R}^N,
\]
and
\begin{equation}
\label{hessian}
D^2 \mathcal{H}(z)=|z|^{p-2}\,\mathrm{Id}+(p-2)\,|z|^{p-4}\,z\otimes z,\qquad \mbox{ for every } z\in\mathbb{R}^N.
\end{equation}
In particular, we have the following facts for $p>2$
\begin{equation}
\label{high}
 D^2\mathcal{H}(z)\,z=(p-1)\,|z|^{p-2}\,z,\qquad |z|^{p-2}\,|\xi|^2\le \langle D^2\mathcal{H}(z)\,\xi,\xi\rangle\le (p-1)\,|z|^{p-2}\,|\xi|^2,\qquad \mbox{ for } z,\xi\in\mathbb{R}^N.
\end{equation}
We will repeatedly use the following elementary inequality.
\begin{lm}
\label{lm:many}
Let $2<p<\infty$ and let $\Omega\subset\mathbb{R}^N$ be an open set. For every $v,w,\varphi\in W^{1,1}_{\rm loc}(\Omega)$, we have 
\[
\left|\langle D^2\mathcal{H}(\nabla \varphi)\,\nabla v,\nabla v\rangle-\langle D^2\mathcal{H}(\nabla \varphi)\,\nabla w,\nabla w\rangle\right|\le (p-1)\,|\nabla \varphi|^{p-2}\,|\nabla v-\nabla w|\,\Big(|\nabla v|+|\nabla w|\Big),\ \mbox{ a.\,e. on }\Omega.
\]
\end{lm}
\begin{proof}
By Lemma \ref{lm:retarded}, we have 
\[
\left|\langle D^2\mathcal{H}(\nabla\varphi)\,\nabla v,\nabla v\rangle-\langle D^2\mathcal{H}(\nabla \varphi)\,\nabla w,\nabla w\rangle\right|\le |D^2\mathcal{H}(\nabla \varphi)\,(\nabla v-\nabla w)|\,\Big(|\nabla v|+|\nabla w|\Big).
\]
By using that the Hessian matrix is given by \eqref{hessian},
we get 
\[
|D^2\mathcal{H}(\nabla \varphi)\,(\nabla v-\nabla w)|\le (p-1)\,|\nabla \varphi|^{p-2}\,|\nabla v-\nabla w|.
\]
By using this inequality in the first estimate, we conclude the proof.
\end{proof}
\begin{prop}
\label{prop:primo}
Let $2<p<\infty$ and let $\Omega\subset\mathbb{R}^N$ be an open connected set with finite volume. Let $u\in W^{1,p}_0(\Omega)$ be the unique positive extremal of 
\begin{equation}
\label{lambda1}
\lambda_p(\Omega)=\min_{\varphi\in W^{1,p}_0(\Omega)}\left\{\int_\Omega |\nabla \varphi|^p\, dx\, :\, \int_\Omega |\varphi|^p=1\right\}.
\end{equation}
By setting 
\[
\lambda(\Omega;u)=\inf_{\varphi\in C^\infty_0(\Omega)}\left\{\int_\Omega \langle D^2\mathcal{H}(\nabla u)\,\nabla \varphi,\nabla \varphi\rangle\,dx\, :\, \int_\Omega u^{p-2}\,|\varphi|^2\,dx=1\right\},
\] 
we have
\[
\lambda(\Omega;u)= (p-1)\,\lambda_p(\Omega).
\]
\end{prop}
\begin{proof}
The inequality
\begin{equation}
\label{uno}
\lambda(\Omega;u)\le (p-1)\, \lambda_p(\Omega),
\end{equation}
is straightforward. Indeed, for every $\varepsilon>0$, we take a non-negative $u_\varepsilon\in C^\infty_0(\Omega)$ such that
\[
\int_\Omega |\nabla u_\varepsilon|^p\,dx<\lambda_p(\Omega)+\varepsilon\qquad \mbox{ and }\qquad \int_\Omega u_\varepsilon^p\,dx=1.
\] 
For such a function, we have
\[
\lim_{\varepsilon\to 0} \int_\Omega |\nabla u_\varepsilon-\nabla u|^p\,dx=0,
\]
and thus in particular we have convergence in $L^p(\Omega)$, as well.
We wish to use the function 
\[
\widetilde{u}_\varepsilon:=\frac{u_\varepsilon}{\displaystyle\left(\int_\Omega u^{p-2}\,u_\varepsilon^2\,dx\right)^\frac{1}{2}},
\]
as a competitor in the problem defining $\lambda(\Omega;u)$.
At this aim, it is not difficult to see that 
\begin{equation}
\label{odilon}
\lim_{\varepsilon\to 0} \int_\Omega u^{p-2}\,\widetilde{u}_\varepsilon^2\,dx=\int_\Omega u^p\,dx=1.
\end{equation}
We can also prove that 
\begin{equation}
\label{odilon2}
\lim_{\varepsilon\to 0} \int_\Omega \langle D^2 \mathcal{H}(\nabla u)\,\nabla u_\varepsilon,\nabla u_\varepsilon\rangle\,dx=\int_\Omega \langle D^2 \mathcal{H}(\nabla u)\,\nabla u,\nabla u\rangle\,dx.
\end{equation}
Indeed, by applying Lemma \ref{lm:many}, we get
\[
\begin{split}
\Bigg| \int_\Omega \langle D^2 \mathcal{H}(\nabla u)\,\nabla u_\varepsilon,\nabla u_\varepsilon\rangle\,dx&- \int_\Omega \langle D^2 \mathcal{H}(\nabla u)\,\nabla u,\nabla u\rangle\,dx\Bigg|\\
&\le (p-1)\,\int_\Omega |\nabla u|^{p-2}\,|\nabla u_\varepsilon-\nabla u|\,\Big(|\nabla u_\varepsilon|+|\nabla u|\Big)\,dx.
\end{split}
\]
Then \eqref{odilon2} follows by using H\"older's inequality.
\par
By using the function $\widetilde u_\varepsilon$, taking the limit as $\varepsilon$ go to $0$, using \eqref{odilon} and \eqref{odilon2} and finally recalling that thanks to \eqref{high} we have
\[
\langle D^2 \mathcal{H}(z)\,z,z\rangle=(p-1)\,|z|^p,\qquad \mbox{ for every }z\in\mathbb{R}^N,
\] 
we get
\[
\lambda(\Omega;u)\le (p-1)\,\lambda_p(\Omega).
\]
Thus \eqref{uno} is established.
\vskip.2cm\noindent
For the converse inequality, we first recall that $u$ satisfies
\begin{equation}
\label{ELu}
\int_\Omega \langle |\nabla u|^{p-2}\,\nabla u,\nabla \varphi\rangle\,dx=\lambda_p(\Omega)\,\int_\Omega u^{p-1}\,\varphi\,dx, \qquad \mbox{ for every } \varphi\in W^{1,p}_0(\Omega),
\end{equation}
by minimality. By using \eqref{high},
this can be also rewritten as
\begin{equation}
\label{ELubis}
\int_\Omega \langle D^2\mathcal{H}(\nabla u)\,\nabla u,\nabla \varphi\rangle\,dx=(p-1)\,\lambda_p(\Omega)\,\int_\Omega u^{p-1}\,\varphi\,dx, \qquad \mbox{ for every } \varphi\in W^{1,p}_0(\Omega).
\end{equation}
We take $\varepsilon>0$ and $\varphi_\varepsilon\in C^\infty_0(\Omega)$ such that
\[
\int_\Omega \langle D^2\mathcal{H}(\nabla u)\,\nabla \varphi_\varepsilon,\nabla \varphi_\varepsilon\rangle\,dx<\lambda(\Omega;u)+\varepsilon\qquad \mbox{ and }\qquad \int_\Omega u^{p-2}\,\varphi_\varepsilon^2\,dx=1.
\]
Then we insert the test function\footnote{This is admissible, since $\varphi$ is compactly supported in $\Omega$, while by the minimum principle $u$ is bounded away from zero on every $\Omega'\Subset\Omega$.} $\varphi_\varepsilon^2/u$ in the equation \eqref{ELubis}, so to get
\begin{equation}
\label{prepicone}
\begin{split}
(p-1)\,\lambda_{p}(\Omega)\,\int_\Omega u^{p-1}\,\frac{\varphi^2_\varepsilon}{u}\,dx&=\int_\Omega  \left\langle D^2\mathcal{H}(\nabla u)\,\nabla u,\nabla\left(\frac{\varphi_\varepsilon^2}{u}\right)\right\rangle\,dx.\\
\end{split}
\end{equation} 
We now use Picone's identity of Lemma \ref{lm:picone} with the choice $A=D^2\mathcal{H}(\nabla u)$.
This gives
\[
\begin{split}
\left\langle D^2\mathcal{H}(\nabla u)\,\nabla u,\nabla\left(\frac{\varphi_\varepsilon^2}{u}\right)\right\rangle&=\langle D^2\mathcal{H}(\nabla u)\,\nabla \varphi_\varepsilon,\nabla \varphi_\varepsilon\rangle\\
&-\left\langle D^2\mathcal{H}(\nabla u)\,\left(\varphi_\varepsilon\,\frac{\nabla u}{u}-\nabla \varphi_\varepsilon\right),\left(\varphi_\varepsilon\,\frac{\nabla u}{u}-\nabla \varphi_\varepsilon\right)\right\rangle.
\end{split}
\]
By integrating over $\Omega$ and using the resulting identity in \eqref{prepicone}, we get
\[
\begin{split}
(p-1)\,\lambda_{p}(\Omega)\,\int_\Omega u^{p-2}\,\varphi^2_\varepsilon\,dx&=\int_\Omega \langle D^2 \mathcal{H}(\nabla u)\,\nabla \varphi_\varepsilon,\nabla \varphi_\varepsilon\rangle\,dx\\
&-\int_\Omega \left\langle D^2\mathcal{H}(\nabla u)\,\left(\varphi_\varepsilon\,\frac{\nabla u}{u}-\nabla \varphi_\varepsilon\right),\left(\varphi_\varepsilon\,\frac{\nabla u}{u}-\nabla \varphi_\varepsilon\right)\right\rangle\,dx\\
&\leq \lambda(\Omega;u)+\varepsilon,
\end{split}
\]
thanks to the fact that $D^2 \mathcal{H}(\nabla u)$ is positive semidefinite.
By recalling that 
\[
\int_\Omega u^{p-2}\,\varphi_\varepsilon^2\,dx=1,
\] 
and using the arbitrariness of $\varepsilon>0$, we finally get the desired conclusion.
\end{proof}
\begin{defi}
\label{defi:spaziodimerda}
For $p>2$, with the notations of Proposition \ref{prop:primo}, we define the weighted Sobolev space
\[
X^{1,2}(\Omega;|\nabla u|^{p-2}):=\left\{\varphi\in W^{1,1}_{\rm loc}(\Omega)\cap L^2(\Omega)\, :\, \int_\Omega |\nabla u|^{p-2}\,|\nabla \varphi|^2\,dx<+\infty\right\},
\]
endowed with the natural norm
\[
\|\varphi\|_{X^{1,2}(\Omega;|\nabla u|^{p-2})}=\|\varphi\|_{L^2(\Omega)}
+\left(\int_\Omega |\nabla u|^{p-2}\,|\nabla \varphi|^2\,dx\right)^\frac{1}{2}.
\] 
Accordingly, we set $X^{1,2}_0(\Omega;|\nabla u|^{p-2})$ for the completion of $C^\infty_0(\Omega)$ with respect to this norm, as in \cite{CES} (see also \cite{Ta}).
\end{defi}
\begin{lm}
\label{lm:embeddo} 
Let $p>2$ and let $\Omega\subset\mathbb{R}^N$ be an open bounded connected set, with $C^{1,\alpha}$ boundary, for some $0<\alpha<1$.
With the notations of Proposition \ref{prop:primo}, we have
\[
X^{1,2}(\Omega;|\nabla u|^{p-2})=\left\{\varphi\in W^{1,1}(\Omega)\cap L^2(\Omega)\, :\, \int_\Omega |\nabla u|^{p-2}\,|\nabla \varphi|^2\,dx<+\infty\right\}.
\]
Moreover, we also have
\[
W^{1,p}_0(\Omega)\subset X^{1,2}_0(\Omega;|\nabla u|^{p-2})\subset W^{1,1}_0(\Omega),
\]
with continuous inclusions. Finally, 
\[
X^{1,2}_0(\Omega;|\nabla u|^{p-2})\subset X^{1,2}(\Omega;|\nabla u|^{p-2}).
\]
\end{lm}
\begin{proof}
In order to prove the first fact, it is sufficient to prove that every $\varphi\in X^{1,2}(\Omega;|\nabla u|^{p-2})$ belongs to $W^{1,1}(\Omega)$, as well. This can be done similarly as in the proof of Corollary \ref{coro:nakata}: for every $\Omega'\Subset \Omega$, we have by H\"older's inequality
\[
\int_{\Omega'} |\nabla \varphi|\,dx\le \left(\int_{\Omega} |\nabla u|^{p-2}\,|\nabla \varphi|^2\,dx\right)^\frac{1}{2}\,\left(\int_{\Omega} \frac{1}{|\nabla u|^{p-2}}\,dx \right)^\frac{1}{2},
\]
and observe that the last integrals are finite, thanks to the definition of $X^{1,2}(\Omega;|\nabla u|^{p-2})$ and to Theorem \ref{teo:sweden}. Since $\Omega'\Subset \Omega$ is arbitrary, this shows that $\varphi\in W^{1,1}(\Omega)$, as desired.\par
Let us now come to the second statement. It is sufficient to prove that there exist two constants $C_1,C_2>0$ such that
\[
C_1\,\|\varphi\|_{W^{1,1}(\Omega)}\le \|\varphi\|_{X^{1,2}(\Omega;|\nabla u|^{p-2})}\le C_2\,\|\varphi\|_{W^{1,p}(\Omega)},\qquad \mbox{ for every } \varphi\in C^\infty_0(\Omega).
\]
The estimate on the left-hand side follows from the first part of the proof and the fact that 
\[
\|\varphi\|_{L^1(\Omega)}\le |\Omega|^\frac{1}{2}\,\|\varphi\|_{L^2(\Omega)}\le |\Omega|^\frac{1}{2}\,\|\varphi\|_{X^{1,2}(\Omega;|\nabla u|^{p-2})}.
\]
The second one follows from H\"older's inequality, which permits to infer that 
\[
\begin{split}
\|\varphi\|_{X^{1,2}(\Omega;|\nabla u|^{p-2})}&=\|\varphi\|_{L^2(\Omega)}+ \left(\int_\Omega |\nabla u|^{p-2}\,|\nabla \varphi|^2\,dx\right)^\frac{1}{2}\\
&\le |\Omega|^{\frac{1}{2}-\frac{1}{p}}\,\|\varphi\|_{L^p(\Omega)}
+ \left(\int_\Omega |\nabla u|^p\,dx\right)^\frac{p-2}{2\,p}\,\|\nabla\varphi\|_{L^p(\Omega)}.
\end{split}
\]
Finally, as for the last statement: observe that every $\{\varphi_n\}_{n\in\mathbb{N}}\subset C^\infty_0(\Omega)$ which is a Cauchy sequence with respect to the norm of $X^{1,2}(\Omega;|\nabla u|^{p-2})$ is a Cauchy sequence in the Banach spaces $W^{1,1}_0(\Omega)$ and $L^2(\Omega)$, as well. Thus it converges in these spaces to a function 
\[
\varphi\in W^{1,1}_0(\Omega)\cap L^2(\Omega).
\]
Moreover, by using the strong $L^1$ convergence of the gradients and the fact that $\nabla u\in L^\infty(\Omega)$ by Theorem \ref{teo:uniformi}, we have for every $k\in\mathbb{N}$  
\[
\begin{split}
\int_{\{|\nabla \varphi|\le k\}}|\nabla u|^{p-2}\, |\nabla \varphi|^2\,dx&=\int_{\{|\nabla \varphi|\le k\}}|\nabla u|^{p-2}\, |\nabla \varphi|^2\,dx\\
&+2\,\lim_{n\to\infty} \int_{\{|\nabla \varphi|\le k\}} |\nabla u|^{p-2}\,\langle \nabla \varphi,\nabla \varphi_n-\nabla \varphi\rangle\,dx\\
&\le \liminf_{n\to\infty} \int_{\{|\nabla \varphi|\le k\}} |\nabla u|^{p-2}\,|\nabla \varphi_n|^2\,dx\\
&\le \liminf_{n\to\infty} \int_\Omega |\nabla u|^{p-2}\,|\nabla \varphi_n|^2\,dx\le C.
\end{split}
\]
By taking the limit as $k$ goes to $\infty$, this finally proves that $\varphi\in X^{1,2}(\Omega;|\nabla u|^{p-2})$. This is enough to conclude the proof.
\end{proof}
We can now characterize the extremals for the variational problem which defines $\lambda(\Omega;u)$. The same result is also contained in \cite[Proposition 4.4]{Ta}. We point out that the proof in \cite{Ta} is different and it uses a slightly stronger assumption on the open set.
\begin{prop}\label{prop:unique} 
Let $2<p<\infty$ and let $\Omega\subset\mathbb{R}^N$ be an open bounded connected set, with $C^{1,\alpha}$ boundary, for some $0<\alpha<1$.
With the notations of Proposition \ref{prop:primo}, the infimum $\lambda(\Omega;u)$ is uniquely attained on the space $X^{1,2}_0(\Omega;|\nabla u|^{p-2})$ by the functions $u$ or $-u$.
\end{prop}
\begin{proof}
We first notice that if $v\in X^{1,2}_0(\Omega;|\nabla u|^{p-2})$ and $\{v_n\}_{n\in\mathbb{N}}\subset C^\infty_0(\Omega)$ is such that 
\[
\lim_{n\to\infty} \|v_n-v\|_{X^{1,2}(\Omega;|\nabla u|^{p-2})}=0,
\]
then 
\begin{equation}
\label{hessest}
\lim_{n\to\infty} \int_\Omega \langle D^2\mathcal{H}(\nabla u)\,\nabla v_n,\nabla v_n\rangle\,dx=\int_\Omega \langle D^2\mathcal{H}(\nabla u)\,\nabla v,\nabla v\rangle\,dx.
\end{equation}
Indeed, by Lemma \ref{lm:many} 
\[
\begin{split}
\left|\langle D^2\mathcal{H}(\nabla u)\,\nabla v_n,\nabla v_n\rangle-\langle D^2\mathcal{H}(\nabla u)\,\nabla v,\nabla v\rangle\right|\le (p-1)\,|\nabla u|^{p-2}\,|\nabla v_n-\nabla v|\,|\,\Big(|\nabla v_n|+|\nabla v|\Big).
\end{split}
\]
By integrating over $\Omega$ and using H\"older's inequality, we have
\[
\begin{split}
\Bigg|\int_\Omega \langle D^2\mathcal{H}(\nabla u)\,\nabla v_n,\nabla v_n\rangle\,dx&-\int_\Omega \langle D^2\mathcal{H}(\nabla u)\,\nabla v,\nabla v\rangle\,dx\Bigg|\\
&\le C\,\int_\Omega |\nabla u|^{p-2}\,|\nabla v_n-\nabla v|\,\Big(|\nabla v_n|+|\nabla v|\Big)\,dx\\
&\le C\,\left(\int_\Omega |\nabla u|^{p-2}\,|\nabla v_n-\nabla v|^2\,dx\right)^\frac{1}{2}\\
&\times\left(\int_\Omega |\nabla u|^{p-2}\,\Big(|\nabla v_n|+|\nabla v|\Big)^2\,dx\right)^\frac{1}{2}.
\end{split}
\]
By observing that the last term converges to $0$, we get \eqref{hessest}. Similarly, by using that $u\in L^\infty(\Omega)$, we get that 
\[
\lim_{n\to\infty} \int_\Omega u^{p-2}\,|v_n|^2\,dx=\int_\Omega u^{p-2}\,|v|^2\,dx.
\]
Since $C^\infty_0(\Omega)$ is dense in $X^{1,2}_0(\Omega;|\nabla u|^{p-2})$ by definition, the previous computations show that 
\[
\lambda(\Omega;u)=\inf_{\varphi\in X^{1,2}_0(\Omega;|\nabla u|^{p-2})} \left\{\int_\Omega \langle D^2\mathcal{H}(\nabla u)\,\nabla \varphi,\nabla \varphi\rangle\,dx\, :\, \int_\Omega u^{p-2}\,|\varphi|^2\,dx=1\right\}.
\]
In order to prove that $u$ or $-u$ attain the infimum, it is sufficient to use \eqref{high}.
This entails that
\[
\int_\Omega \langle D^2\mathcal{H}(\nabla u)\,\nabla u,\nabla u\rangle\,dx=(p-1)\,\int_\Omega |\nabla u|^p\,dx=(p-1)\,\lambda_p(\Omega)=\lambda(\Omega;u),
\]
where the last equality is the content of Proposition \ref{prop:primo}.
By further observing that $u\in W^{1,p}_0(\Omega)\subset X^{1,2}_0(\Omega;|\nabla u|^{p-2})$ by Lemma \ref{lm:embeddo}, we get that $u$ and $-u$ are minimizers for the problem defining the value $\lambda(\Omega;u)$.
\vskip.2cm\noindent
In order to prove that any minimizer must coincide either with $u$ or with $-u$, we assume that there is another minimizer $v\in X^{1,2}_0(\Omega;|\nabla u|^{p-2})$. By definition, there exists a sequence $\{v_n\}_{n\in\mathbb{N}}\subset C^\infty_0(\Omega)$ such that
\[
\lim_{n\to\infty} \left[\int_\Omega |\nabla u|^{p-2}\,|\nabla v_n-\nabla v|^2\,dx+\int_\Omega |v_n-v|^2\,dx\right]=0.
\]
We recall that $u$ satisfies \eqref{ELubis}.
For every $n\in\mathbb{N}$, the choice $\varphi =v_n^2/u$ is feasible in \eqref{ELubis} and it yields
\begin{equation}
\label{porconi}
\begin{split}
\lambda(\Omega; u)\,\int_\Omega u^{p-2}\, v_n^2\, dx	
&= \int_\Omega \left\langle D^2\mathcal{H}(\nabla u)\,\nabla u,\nabla\left(\frac{v_n^2}{u}\right)\right\rangle\,dx\\
&= \int_\Omega \langle D^2\mathcal{H}(\nabla u)\,\nabla v_n,\nabla v_n\rangle\,dx\\
&-\int_\Omega \left\langle D^2\mathcal{H}(\nabla u)\,\left(v_n\,\frac{\nabla u}{u}-\nabla v_n\right),\left(v_n\,\frac{\nabla u}{u}-\nabla v_n\right)\right\rangle\,dx,\\
\end{split}
\end{equation}
where in the second equality we used the general version of the Picone identity given by Lemma \ref{lm:picone}, with the positive semidefinite matrix $A=D^2\mathcal{H}(\nabla u)$. 
\par
We now wish to pass to the limit as $n$ goes to $\infty$ in the previous identity. We notice at first that 
\[
\lim_{n\to\infty}\int_\Omega u^{p-2}\, v_n^2\, dx=\int_\Omega u^{p-2}\, v^2\, dx=1,
\]
which follows directly by the choice of $\{v_n\}_{n\in\mathbb{N}}$ and the fact that $u\in L^\infty(\Omega)$. As for the first term on the right-hand side of \eqref{porconi}, we simply use \eqref{hessest}.
\par
We are left with handling the last term in \eqref{porconi}. We have that $\{(v_n,\nabla v_n)\}_{n\in\mathbb{N}}$ converges almost everywhere to $(v,\nabla v)$, possibly up to extracting a subsequence. Observe that we are using that $|\nabla u|\not=0$ almost everywhere in $\Omega$, thanks to \eqref{DS}.
By observing that $D^2 \mathcal{H}(\nabla u)$ is positive semidefinite, an application of Fatou's Lemma yields
\[
\begin{split}
\liminf_{n\to \infty} &\int_\Omega \left\langle D^2\mathcal{H}(\nabla u)\,\left(v_n\,\frac{\nabla u}{u}-\nabla v_n\right),\left(v\,\frac{\nabla u}{u}-\nabla v_n\right)\right\rangle\,dx\\
&\ge \int_\Omega \left\langle D^2\mathcal{H}(\nabla u)\,\left(v\,\frac{\nabla u}{u}-\nabla v\right),\left(v\,\frac{\nabla u}{u}-\nabla v\right)\right\rangle\,dx.
\end{split}
\]
Thus, by taking the limit as $n$ goes to $\infty$ in \eqref{porconi}, we get
\[
\begin{split}
\lambda(\Omega; u)&+\int_\Omega \left\langle D^2\mathcal{H}(\nabla u)\,\left(v\,\frac{\nabla u}{u}-\nabla v\right),\left(v\,\frac{\nabla u}{u}-\nabla v\right)\right\rangle\,dx\le \lambda(\Omega; u).
\end{split}
\] 
This entails that we must have
\begin{equation}
\label{condizione}
\left\langle D^2\mathcal{H}(\nabla u)\,\left(v\,\frac{\nabla u}{u}-\nabla v\right),\left(v\,\frac{\nabla u}{u}-\nabla v\right)\right\rangle=0,\qquad \mbox{a.\,e. in }\Omega.
\end{equation}
From the definition of $D^2 \mathcal{H}$, 
it is clear that $D^2\mathcal{H}(\nabla u)$ is positive definite whenever $\nabla u$ does not vanish: as remarked above, this is true almost everywhere by \eqref{DS}. Therefore, from \eqref{condizione} we must have
\begin{equation}
\label{watchout}
v\frac{\nabla u}{u}-\nabla v = 0,\qquad \mbox{ a.\,e. in }\Omega.
\end{equation}
We now observe that $v\in W^{1,1}_0(\Omega)$ thanks to Lemma \ref{lm:embeddo}, while by Theorem \ref{teo:uniformi} $u\in C^1(\overline\Omega)$ and it has the following property: for every $\Omega'\Subset \Omega$, there exists a constant $C=C(\Omega')>0$ such that $u\ge 1/C$ on $\Omega'$. Thus we have 
\[
\frac{v}{u}\in W^{1,1}_{\rm loc}(\Omega),
\]
and Leibniz's rule holds for its distributional gradient. The latter is given by 
\[
\nabla \left(\frac{v}{u}\right)=\frac{u\,\nabla v-v\,\nabla u}{u^2},\qquad \mbox{ a.\,e. in }\Omega,
\]
and thus it identically vanishes almost everywhere in $\Omega$, by virtue of \eqref{watchout}. 
Since $\Omega$ is connected, this implies that $v/u$ is constant in $\Omega$. Thus we get that $v$ is proportional to $u$ in $\Omega$. The desired result is now a consequence of the normalization taken.
\end{proof}

\section{Proofs of the main results} 
\label{sec:proof}

\begin{proof}[Proof of Theorem \ref{thm:main}]
We divide the proof in three parts, for ease of readability.
\vskip.2cm\noindent
{\it Part 1: linearized equation}.
We argue by contradiction: we suppose that for every $q>p$, the problem \eqref{eq:pq} always admits (at least) two distinct positive solutions. We then take a decreasing sequence $\{q_n\}_{n\in\mathbb{N}}\subset (p,+\infty)$ such that 
\[
\lim_{n\to\infty} q_n=p.
\]
Correspondingly, for every $n\in\mathbb{N}$ there exist two distinct positive solutions $u_n$ and $v_n$ of \eqref{eq:pq}. Observe that they solve
\[
-\Delta_p u_n=\lambda_{p,q_n}(\Omega)\,u_n^{q_n-1}\quad \mbox{ and }\quad -\Delta_p v_n=\lambda_{p,q_n}(\Omega)\,v_n^{q_n-1},\qquad \mbox{ in }\Omega,
\]
with 
\[
\int_\Omega u_n^{q_n}\,dx=\int_\Omega v_n^{q_n}\,dx=1.
\]
By recalling that $q\mapsto \lambda_{p,q}(\Omega)$ is continuous, we get
\[
\lim_{n\to \infty} \lambda_{p,q_n}(\Omega)=\lambda_p(\Omega).
\]
Then, by using the minimality and the uniform convexity of the $L^p$ norm, it is not difficult to see that
\[
\lim_{n\to\infty} \|u_n-u\|_{W^{1,p}_0(\Omega)}=\lim_{n\to\infty} \|v_n-u\|_{W^{1,p}_0(\Omega)}=0,
\]
where $u\in W^{1,p}_0(\Omega)$ is the unique positive solution of \eqref{lambda1}. In turn, such a convergence can be upgraded to a convergence in $C^1(\overline\Omega)$ norm, thanks to the uniform $C^{1,\chi}$ estimate of Theorem \ref{teo:uniformi}.
\par
If we recall the notation
\[
\mathcal{H}(z)=\frac{1}{p}\,|z|^p,
\] 
the equations solved by $u_n$ and $v_n$ can be written in weak form as
\begin{equation}
\label{u_n}
\int_\Omega \langle \nabla \mathcal{H}(\nabla u_n),\nabla\varphi\rangle=\lambda_{p,q_n}(\Omega)\,\int_\Omega u_n^{q_n-1}\,\varphi\,dx,
\end{equation}
and
\begin{equation}
\label{v_n}
\int_\Omega \langle \nabla \mathcal{H}(\nabla v_n),\nabla\varphi\rangle=\lambda_{p,q_n}(\Omega)\,\int_\Omega v_n^{q_n-1}\,\varphi\,dx,
\end{equation}
for any $\varphi\in W^{1,p}_0(\Omega)$.

We now observe that for every $z,w\in\mathbb{R}^N$ we have
\begin{equation}
\label{gradienti}
\begin{split}
\nabla \mathcal{H}(z)-\nabla \mathcal{H}(w)&=\int_0^1 \frac{d}{dt}\Big( \nabla\mathcal{H}(t\,z+(1-t)\,w\Big)\,dt\\
&=\left(\int_0^1 D^2\mathcal{H}(t\,z+(1-t)\,w)\,dt\right)\,(z-w).
\end{split}
\end{equation}
Similarly, for every $a,b\ge 0$ we have
\begin{equation}
\label{funzioni}
\begin{split}
a^{q_n-1}-b^{q_n-1}&=\int_0^1 \frac{d}{dt} (t\,a+(1-t)\,b)^{q_n-1}\,dt\\
&=(q_n-1)\,\left(\int_0^1 (t\,a+(1-t)\,b)^{q_n-2}\,dt\right)\,(a-b).
\end{split}
\end{equation}
By subtracting the two equations \eqref{u_n} and \eqref{v_n}, using \eqref{gradienti} with $z=\nabla u_n(x)$, $w=\nabla v_n(x)$ and \eqref{funzioni} with $a=u_n(x)$, $b=v_n(x)$,
we thus get
\begin{equation}
\label{almostlinear}
\begin{split}
\int_\Omega \langle A_{n}(x)\,\nabla (u_n-v_n),\nabla \varphi\rangle\,dx&=\lambda_{p,q_n}(\Omega)\,\int_\Omega w_n\,(u_n-v_n)\,\varphi\,dx\\
\end{split}
\end{equation}
where 
\[
A_n(x)=\int_0^1 D^2\mathcal{H}(t\,\nabla u_n(x)+(1-t)\,\nabla v_n(x))\,dt,
\]
and 
\[
w_n(x)=(q_n-1)\,\int_0^1 (t\,u_n(x)+(1-t)\,v_n(x))^{q_n-2}\,dt.
\]
For every $n\in\mathbb{N}$, we set
\[
\phi_n=\frac{u_n-v_n}{\|u_n-v_n\|_{L^2(\Omega)}}\in W^{1,p}_0(\Omega),
\]
then from \eqref{almostlinear} we get that $\phi_n$ solves the following weighted linear eigenvalue problem
\begin{equation}
\label{linearized}
\int_\Omega \langle A_n(x)\,\nabla\phi_n,\nabla \varphi\rangle\,dx=\lambda_{p,q_n}(\Omega)\,\int_\Omega w_n\,\phi_n\,\varphi\,dx,\qquad \mbox{ for } \varphi\in W^{1,p}_0(\Omega). 
\end{equation}
In particular, the choice $\varphi=\phi_n$ in \eqref{linearized} yields
\begin{equation}
\label{neqn}
\int_\Omega \langle A_n\,\nabla \phi_n,\nabla\phi_n\rangle\,dx=\lambda_{p,q_n}(\Omega)\,\int_\Omega w_n\,|\phi_n|^2\,dx.
\end{equation}
\noindent
{\it Part 2: convergence of $\phi_n$}. We now would like to know that it is possible to pass to the limit in \eqref{linearized} and \eqref{neqn}.
By observing that 
\begin{equation}
\label{eq:wbound}
\|\phi_n\|_{L^2(\Omega)}=1\qquad \mbox{ and }\qquad \lambda_{p,q_n}(\Omega)\,\|w_n\|_{L^\infty(\Omega)}\le C,
\end{equation}
inequality \eqref{neqn} implies
\[
\int_\Omega \langle A_n\,\nabla \phi_n,\nabla\phi_n\rangle\,dx\le C.
\]
Observe that the second uniform bound in \eqref{eq:wbound} can be inferred from Proposition \ref{prop:Linfty}, the properties of $u_n,v_n$ and \eqref{lowerlambda}.
An application of \eqref{high} and Lemma \ref{lm:integral} yields
\begin{equation}
\label{eq:ralle}
\begin{split}
\langle A_n\,\xi,\xi\rangle&=\int_0^1 \langle D^2\mathcal{H}(t\,\nabla u_n+(1-t)\,\nabla v_n)\,\xi,\xi\rangle\,dt\\
&\ge \left(\int_0^1|t\,\nabla u_n+(1-t)\,\nabla v_n|^{p-2}\,dt\right)\,|\xi|^2\ge \frac{1}{4^{p-1}}\, (|\nabla u_n|+|\nabla v_n|)^{p-2}\,|\xi|^2.
\end{split}
\end{equation}
Therefore, we obtain
\[
\int_\Omega (|\nabla u_n|+|\nabla v_n|)^{p-2}\,|\nabla \phi_n|^2\,dx\le C,\qquad \mbox{ for every } n\in\mathbb{N}.
\]
By Corollary \ref{lem:compact}, there exist $\theta=\theta(p)\in (1,2)$ and $\phi\in L^2(\Omega)\cap W^{1,\theta}_0(\Omega)$  such that $\{\phi_n\}_{n\in\mathbb{N}}$ converges strongly in $L^2(\Omega)$ and weakly in $W^{1,\theta}(\Omega)$ to $\phi$, up to a subsequence. 
In addition, from the $C^1(\overline\Omega)$ convergence of $u_n$ and $v_n$, we have
\[
A_n\to  D^2\mathcal{H}(\nabla u)\qquad \mbox{ and }\qquad \lambda_{p,q_n}(\Omega)\,w_n\to (p-1)\,\lambda_p(\Omega)\, u^{p-2}\qquad \mbox{ uniformly on } \overline\Omega.
\] 
This is enough to pass to the limit in \eqref{linearized} and \eqref{neqn}, as we will now see.
Indeed, the convergence of the right-hand side of \eqref{linearized} easily follows from the claimed convergences. As for the left-hand side, we have for every $\varphi\in C^\infty_0(\Omega)$
\[
\begin{split}
\Bigg|\int_\Omega \langle A_n(x)\,\nabla\phi_n,\nabla \varphi\rangle\,dx&-\int_\Omega \langle D^2\mathcal{H}(\nabla u)\,\nabla\phi,\nabla \varphi\rangle\,dx\Bigg|\\
&\le \Bigg|\int_\Omega \left\langle \Big(A_n(x)-D^2\mathcal{H}(\nabla u)\Big)\,\nabla\phi_n,\nabla \varphi\right\rangle\,dx\Bigg|\\
&+\Bigg|\int_\Omega \langle D^2\mathcal{H}(\nabla u)\,(\nabla\phi_n-\nabla \phi),\nabla \varphi\rangle\,dx\Bigg|\\
&\le \|A_n-D^2\mathcal{H}(\nabla u)\|_{L^\infty(\Omega)}\,\|\nabla \varphi\|_{L^\infty(\Omega)}\,\int_\Omega |\nabla \phi_n|\,dx \\
&+\Bigg|\int_\Omega \langle D^2\mathcal{H}(\nabla u)\,(\nabla\phi_n-\nabla \phi),\nabla \varphi\rangle\,dx\Bigg|.
\end{split}
\]
The first term converges to zero thanks to the uniform convergence of $A_n$ and the uniform bound on $\{\phi_n\}_{n\in\mathbb{N}}$ in $W^{1,\theta}(\Omega)$. As for the second term, it is sufficient to use that $D^2\mathcal{H}(\nabla u)\in L^\infty(\Omega)$ and the weak convergence of the gradients of $\{\phi_n\}_{n\in\mathbb{N}}$.

We thus obtain that $\phi$ satisfies
\[
\int_\Omega \langle D^2\mathcal{H}(\nabla u)\,\nabla \phi,\nabla \varphi\rangle\,dx=(p-1)\,\int_\Omega u^{p-2}\,\phi\,\varphi\,dx,\qquad \mbox{ for every }\varphi\in C^\infty_0(\Omega).
\]
In order to pass to the limit in \eqref{neqn}, we observe that\footnote{In the second inequality, we use the ``above tangent'' property 
\[
f(z)\ge f(z_0)+\langle \nabla f(z_0),z-z_0\rangle,\qquad \mbox{ for every } z,z_0\in\mathbb{R}^N,
\]
for the convex function $f(z)=\langle A_n\,z,z\rangle$.} for every $n,k\in\mathbb{N}$
\[
\begin{split}
\int_\Omega \langle A_n\,\nabla \phi_n,\nabla\phi_n\rangle\,dx&\ge \int_{\{|\nabla \phi|\le k\}} \langle A_n\,\nabla \phi_n,\nabla\phi_n\rangle\,dx\\
&\ge  \int_{\{|\nabla \phi|\le k\}} \langle A_n\,\nabla \phi,\nabla\phi\rangle\,dx+2\, \int_{\{|\nabla \phi|\le k\}} \langle A_n\,\nabla \phi,\nabla\phi_n-\nabla \phi\rangle\,dx.
\end{split}
\]
By using the weak convergence of $\nabla \phi_n$ and the fact $A_n\,\nabla \phi$ is uniformly bounded on $\{|\nabla \phi|\le k\}$ for every fixed $k$, we get
\[
\lim_{n\to\infty}  \int_{\{|\nabla \phi|\le k\}} \langle A_n\,\nabla \phi,\nabla\phi_n-\nabla \phi\rangle\,dx=0.
\]
This implies that for every $k\in\mathbb{N}$ we have
\[
\begin{split}
\liminf_{n\to\infty}
\int_\Omega \langle A_n\,\nabla \phi_n,\nabla\phi_n\rangle\,dx&\ge \liminf_{n\to\infty}\int_{\{|\nabla \phi|\le k\}} \langle A_n\,\nabla \phi,\nabla\phi\rangle\,dx\\
&=\int_{\{|\nabla \phi|\le k\}} \langle D^2\mathcal{H}(\nabla u)\,\nabla \phi,\nabla\phi\rangle\,dx,
\end{split}
\]
thanks to the uniform convergence of $A_n$. We can now take the limit as $k$ goes to $\infty$ and obtain that the left-hand side of \eqref{neqn} is lower semicontinuous. Thus we obtain 
\begin{equation}
\label{eppo}
\int_\Omega \langle D^2\mathcal{H}(\nabla u)\,\nabla \phi,\nabla\phi\rangle\,dx\leq (p-1)\,\lambda_{p}(\Omega)\,\int_\Omega u^{p-2}\,|\phi|^2\,dx.
\end{equation}
Observe that in the right-hand side we used the strong convergence in $L^2(\Omega)$ of $\{\phi_n\}_{n\in\mathbb{N}}$. By recalling that 
\[
\langle D^2\mathcal{H}(z)\,\xi,\xi\rangle\ge |z|^{p-2}\,|\xi|^2,\qquad \mbox{ for every } z,\xi\in \mathbb{R}^N,
\]
the estimate \eqref{eppo} shows that $\phi$ also belongs to the weighted Sobolev space $X^{1,2}(\Omega;|\nabla u|^{p-2})$ (recall the Definition \ref{defi:spaziodimerda} above).  Note also that the strong convergence of $\phi_n$ in $L^2$ together with \eqref{eq:wbound} implies that $\|\phi\|_{L^2(\Omega)}=1$, so that $\phi$ is non-trivial.
\par
Finally, from the properties above we have 
\[
\phi\in X^{1,2}(\Omega;|\nabla u|^{p-2})\cap W^{1,1}_0(\Omega)= X^{1,2}_0(\Omega;|\nabla u|^{p-2}),
\]
thanks to Lemma \ref{lm:bonucci}.
\vskip.2cm\noindent
{\it Part 3: conclusion.} From the fact that $\phi\in X^{1,2}_0(\Omega;|\nabla u|^{p-2})$ is nontrivial together with
Proposition \ref{prop:primo}, Proposition \ref{prop:unique} and \eqref{eppo}, it follows that $\phi$ must be proportional either to $u$ or to $-u$. In particular, $\phi$ does not change sign: more precisely, it is either strictly negative or strictly positive.
\par
On the other hand, by Lemma \ref{lm:nonempty}, we know that $u_n-v_n$ must change sign. Accordingly, if $\phi_n^\pm$ stand for the positive and negative part of $\phi_n$ respectively, we have that each
\[
\Omega_n^\pm = \{x\in \Omega_n: \, \phi^\pm_n(x)>0\},
\]
has positive measure. Testing equation \eqref{linearized} with  $\phi_n^\pm$, we obtain by using \eqref{eq:wbound}
\begin{equation}\label{eq:ralle2}
\int_\Omega \langle A_n\,\nabla \phi_n^\pm,\nabla\phi_n^\pm\rangle\,dx=\lambda_{p,q_n}(\Omega)\,\int_\Omega w_n\,|\phi_n^\pm|^2\,dx\le C\,\int_\Omega |\phi_n^\pm|^2\,dx.
\end{equation}
By H\"older's inequality, Theorem \ref{teo:WSI}, equations \eqref{eq:ralle} and \eqref{eq:ralle2} we have for an exponent $2<\sigma<\sigma_0$
\[
\begin{split}
\int_\Omega |\phi_n^\pm|^2\,dx &\leq \left(\int_\Omega |\phi_n^\pm|^\sigma dx\right)^\frac{2}{\sigma}|\Omega_n^\pm|^\frac{\sigma-2}{\sigma}\\
&\leq \frac{1}{\mathcal{T}}\,  |\Omega_n^\pm|^\frac{\sigma-2}{\sigma}\int_\Omega|\nabla u_n|^{p-2}\,|\nabla \phi_n^\pm|^2\,dx\\
&\leq \frac{4^{p-1}}{\mathcal{T}}\,|\Omega_n^\pm|^\frac{\sigma-2}{\sigma}\int_\Omega \langle A_n\,\nabla \phi_n^\pm,\nabla\phi_n^\pm\rangle\,dx \le C\,\frac{4^{p-1}}{\mathcal{T}}\,|\Omega_n^\pm|^\frac{\sigma-2}{\sigma}\int_\Omega |\phi_n^\pm|^2\,dx.
\end{split}
\]
This implies
\[
|\Omega_n^\pm|\geq \frac{1}{\widetilde C},\qquad \mbox{ for every } n\in\mathbb{N},
\]
for some constant $\widetilde C$, not depending on $n$.
This contradicts the fact, shown in {\it Part 2}, that $\phi_n$ strongly converges in $L^2(\Omega)$ to the function $\phi$, the latter being either strictly positive or strictly negative. Such a contradiction can be obtained by reasoning as in the proof of \cite[Theorem 1]{Er2014}, for example. The proof is over.
\end{proof}
\begin{oss}
We have already noticed that the function $q\mapsto\lambda_{p,q}(\Omega)$ is continuous. Actually, such a function is $C^1$ on each interval for which $\lambda_{p,q}(\Omega)$ is simple, see \cite{Er2}. In particular, under the standing assumptions of Theorem \ref{thm:main}, we get that such a function is $C^1$ on $[1,\overline{q})$, where $\overline{q}>p$ is as in the statement. We owe this observation to the kind courtesy of G. Ercole.
\end{oss}

\begin{proof}[Proof of Corollary \ref{coro:positive}]
Let $v\in W^{1,p}_0(\Omega)\setminus\{0\}$ be a critical point of the functional $\mathfrak{F}_{q,\lambda}$. By definition of $\lambda_{p,q}(\Omega)$, we have
\[
\lambda_{p,q}(\Omega)\le \frac{\displaystyle\int_\Omega |\nabla v|^p\,dx}{\displaystyle\left(\int_\Omega |v|^q\,dx\right)^\frac{p}{q}}=\lambda\,\left(\int_\Omega |v|^q\,dx\right)^\frac{q-p}{q},
\]
that is 
\begin{equation}
\label{lowerbound}
\left(\frac{\lambda_{p,q}(\Omega)}{\lambda}\right)^\frac{q}{q-p}\le\int_\Omega |v|^q\,dx.
\end{equation}
Moreover, we obtain that equality holds in \eqref{lowerbound} if and only if $v/\|v\|_{L^q(\Omega)}$ is an extremal for \eqref{eq:pq}. On the other hand, if $u\in W^{1,p}_0(\Omega)$ is the unique positive minimizer of \eqref{eq:pq}, it is easily seen that 
\begin{equation}
\label{U}
U=\left(\frac{\lambda}{\lambda_{p,q}(\Omega)}\right)^\frac{1}{p-q}\,u,
\end{equation}
is a critical point of $\mathfrak{F}_{q,\lambda}$ and equality in \eqref{lowerbound} holds. This finally shows that 
\[
\left(\frac{\lambda_{p,q}(\Omega)}{\lambda}\right)^\frac{q}{q-p}=\inf\left\{\int_\Omega |v|^q\,dx\, :\, v\in W^{1,p}_0(\Omega)\setminus\{0\} \mbox{ is a critical point of } \mathfrak{F}_{q,\lambda}\right\},
\]
and such an infimum is uniquely attained at the critical point \eqref{U}.
\end{proof}

\appendix

\section{Inequalities}
\label{sec:A}

In Section \ref{sec:3}, we used the following generalization of Picone's identity for computing the first eigenvalue of the linearized operator.
\begin{lm}[Picone--type identity]
\label{lm:picone}
Let $u,v:\Omega\to \mathbb{R}$ be two differentiable functions, such that $u>0$ in $\Omega$. Let $A$ be an $N\times N$ symmetric matrix with real coefficients. Then we have
\begin{equation}
\label{picone}
\left\langle A\,\nabla u,\nabla\left(\frac{v^2}{u}\right)\right\rangle=\langle A\,\nabla v,\nabla v\rangle-\left\langle A\,\left(v\,\frac{\nabla u}{u}-\nabla v\right),\left(v\,\frac{\nabla u}{u}-\nabla v\right)\right\rangle.
\end{equation}
In particular, if $A$ is positive semidefinite, we get
\[
\left\langle A\,\nabla u,\nabla\left(\frac{v^2}{u}\right)\right\rangle\le \langle A\,\nabla v,\nabla v\rangle
\]
Finally, if $A$ is positive definite, equality in the previous estimate holds if and only if
\[
v\,\frac{\nabla u}{u}=\nabla v.
\]
\end{lm}
\begin{proof}
The proof of \eqref{picone} is by direct computation and is left to the reader. The other facts easily follow from \eqref{picone} and the additional properties of $A$.
\end{proof}
\begin{lm}
\label{lm:retarded}
Let $A$ be an $N\times N$ symmetric matrix with real coefficients. For every $z,w\in\mathbb{R}^N$ we have 
\[
\Big|\langle A\,z,z\rangle-\langle A\,w,w\rangle\Big|\le |A\,(z-w)|\,\Big(|z|+|w|\Big).
\]
\end{lm}
\begin{proof}
We set $z_t=(1-t)\,w+t\,z$ for $t\in[0,1]$, then by basic Calculus we have 
\[
\begin{split}
\Big|\langle A\,z,z\rangle-\langle A\,w,w\rangle\Big|=\left|\int_0^1\frac{d}{dt} \langle A\,z_t,z_t\rangle\,dt\right|&=2\,\left|\int_0^1\langle A\,(z-w),z_t\rangle\,dt\right|\\
&\le 2\,|A\,(z-w)|\,\int_0^1 |z_t|\,dt\\
&\le  2\,|A\,(z-w)|\,\int_0^1 \Big((1-t)\,|w|+t\,|z|\Big)\,dt.
\end{split}
\]
By computing the last integral, we get the desired conclusion.
\end{proof}
The next estimate is quite standard, we include the proof for completeness.
\begin{lm}
\label{lm:integral}
Let $p> 2$, then for every $z,w\in\mathbb{R}^N$ we have
\[
\int_0^1 |t\,z+(1-t)\,w|^{p-2}\,dt\ge \frac{1}{4^{p-1}}\,(|z|+|w|)^{p-2}.
\]
\end{lm}
\begin{proof}
We first suppose that $|z|>|w|$, then for every $t\ge 3/4$ we have
\[
\begin{split}
|t\,z+(1-t)\,w|\ge t\,|z|-(1-t)\,|w|&\ge (2\,t-1)\,|z|\\
&\ge \frac{1}{2}\,|z|\ge \frac{1}{4}\,(|z|+|w|).
\end{split}
\]
Thus in this case, we get
\[
\int_0^1 |t\,z+(1-t)\,w|^{p-2}\,dt\ge \int_\frac{3}{4}^1 |t\,z+(1-t)\,w|^{p-2}\,dt\ge \frac{1}{4^{p-1}}\,(|z|+|w|)^{p-2}.
\]
Similarly, if $|w|\ge |z|$ and $0\le t\le 1/4$, we have
\[
\begin{split}
|t\,z+(1-t)\,w|\ge (1-t)\,|w|-t\,|z|&\ge (1-2\,t)\,|w|\\
&\ge \frac{1}{2}\,|w|\ge \frac{1}{4}\,(|z|+|w|).
\end{split}
\]
By raising both sides to the power $p-2$ and integrating over $[0,1/4]$, we get the desired conclusion in this case, as well.
\end{proof}
The following interpolation inequality is well-known (see for example \cite[Theorem 12.83]{Le}), we focus here on the dependence of the constant on the exponent $q$.
\begin{prop}[Morrey--type inequality]
\label{prop:morrey}
Let $p>N$ and $1\le q<\infty$. There exists a constant $Q_{N,p}>0$, independent of $q$, such that for every $\varphi\in C^\infty_0(\mathbb{R}^N)$ we have 
\begin{equation}
\label{morrey2}
\|\varphi\|_{L^\infty(\mathbb{R}^N)}\le Q_{N,p}\,\left(\int_{\mathbb{R}^N} |\nabla \varphi|^p\,dx\right)^\frac{N}{p\,q-(q-p)\,N}\,\left(\int_{\mathbb{R}^N} |\varphi|^q\,dx\right)^\frac{p-N}{p\,q-(q-p)\,N}.
\end{equation}
\end{prop}
\begin{proof}
Let us set $\alpha=1-N/p$. From the classical {\it Morrey's inequality},
for every $x,y\in\mathbb{R}^N$ we have
\[
C\,|\varphi(x)-\varphi(y)|\le |x-y|^\alpha\,\left(\int_{\mathbb{R}^N} |\nabla \varphi|^p\,dx\right)^\frac{1}{p},
\]
for some $C=C(N,p)>0$. By using the triangle inequality, we get
\[
C\,|\varphi(x)|
\le |x-y|^\alpha\,\left(\int_{\mathbb{R}^N} |\nabla \varphi|^p\,dx\right)^\frac{1}{p}+C\,|\varphi(y)|.
\]
For every $x\in\mathbb{R}^N$, we integrate the previous inequality with respect to $y\in B_1(x)$. This yields
\begin{equation}
\label{premorrey}
C\,\omega_N\,|\varphi(x)|\le \int_{B_1(x)}|x-y|^\alpha\,dy\,\left(\int_{\mathbb{R}^N} |\nabla \varphi|^p\,dx\right)^\frac{1}{p}+C\,\int_{B_1(x)} |\varphi(y)|\,dy.
\end{equation}
We observe that
\[
\int_{B_1(x)}|x-y|^\alpha\,dy=\frac{N\,\omega_N}{N+\alpha},
\]
while by using H\"older's inequality\footnote{For $q=1$ this is not needed.}
\[
\int_{B_1(x)} |\varphi(y)|\,dy\le |B_1(x)|^{1-\frac{1}{q}}\,\left(\int_{B_1(x)} |\varphi(y)|^q\,dy\right)^\frac{1}{q}\le \omega_N^\frac{q-1}{q}\,\left(\int_{\mathbb{R}^N} |\varphi(y)|^q\,dy\right)^\frac{1}{q}.
\]
By using these facts in \eqref{premorrey}, we finally get
\[
|\varphi(x)|\le \frac{1}{C}\,\frac{N}{N+\alpha}\,\left(\int_{\mathbb{R}^N} |\nabla \varphi|^p\,dx\right)^\frac{1}{p}+\omega_N^{-\frac{1}{q}}\,\left(\int_{\mathbb{R}^N} |\varphi|^q\,dx\right)^\frac{1}{q}.
\]
By arbitrariness of $x\in\mathbb{R}^N$, this proves 
\[
\|\varphi\|_{L^\infty(\mathbb{R}^N)}\le \frac{1}{C}\,\frac{N}{N+\alpha}\,\left(\int_{\mathbb{R}^N} |\nabla \varphi|^p\,dx\right)^\frac{1}{p}+\omega_N^{-\frac{1}{q}}\,\left(\int_{\mathbb{R}^N} |\varphi|^q\,dx\right)^\frac{1}{q}.
\]
Observe that the function $q\mapsto \omega_N^{-1/q}$ is continuous and monotone. Thus, we can infer that
\[
\omega_N^{-\frac{1}{q}}\le \max\left\{\frac{1}{\omega_N},1\right\},\qquad \mbox{ for every } 1\le q<+\infty,
\] 
so to obtain
\begin{equation}
\label{premorrey2}
\|\varphi\|_{L^\infty(\mathbb{R}^N)}\le \widetilde{C}\,\left(\int_{\mathbb{R}^N} |\nabla \varphi|^p\,dx\right)^\frac{1}{p}+\widetilde{C}\,\left(\int_{\mathbb{R}^N} |\varphi|^q\,dx\right)^\frac{1}{q},
\end{equation}
with $\widetilde{C}=\widetilde{C}(N,p)>0$. 
\par
The conclusion now is standard. 
We take $\varphi\in C^\infty_0(\mathbb{R}^N)$ and apply \eqref{premorrey2} to the function $\varphi_t(x)=\varphi(t\,x)$, where $t>0$. By observing that 
\[
\|\varphi_t\|_{L^\infty(\mathbb{R}^N)}=\|\varphi\|_{L^\infty(\mathbb{R}^N)},
\]
and making the change of variables $t\,x=y$, we thus obtain
\[
\|\varphi\|_{L^\infty(\mathbb{R}^N)}\le \widetilde{C}\,t^{1-\frac{N}{p}}\,\left(\int_{\mathbb{R}^N} |\nabla \varphi|^p\,dy\right)^\frac{1}{p}+\widetilde{C}\,t^{-\frac{N}{q}}\,\left(\int_{\mathbb{R}^N} |\varphi|^q\,dy\right)^\frac{1}{q},
\]
which holds for every $t>0$. In particular, we obtain
\[
\|\varphi\|_{L^\infty(\mathbb{R}^N)}\le \widetilde{C}\,\inf_{t>0}\left[t^{1-\frac{N}{p}}\,\left(\int_{\mathbb{R}^N} |\nabla \varphi|^p\,dy\right)^\frac{1}{p}+t^{-\frac{N}{p}}\,\left(\int_{\mathbb{R}^N} |\varphi|^p\,dy\right)^\frac{1}{p}\right].
\]
For $p>N$, it is easily seen that the function
\[
h(t)=t^{1-\frac{N}{p}}\,A+t^{-\frac{N}{q}}\,B,\qquad \mbox{ for }t>0,
\] 
is minimal for
\[
t=\left(\frac{B}{A}\, \frac{p}{p-N}\,\frac{N}{q}\right)^\frac{1}{1-\frac{N}{p}+\frac{N}{q}}.
\]
By using this fact with
\[
A=\left(\int_{\mathbb{R}^N} |\nabla \varphi|^p\,dy\right)^\frac{1}{p}\qquad \mbox{ and }\qquad B=\left(\int_{\mathbb{R}^N} |\varphi|^q\,dy\right)^\frac{1}{q},
\]
and computing the minimum above, we finally get the desired estimate \eqref{morrey2} with
\[
Q_{N,p,q}=\widetilde{C}\,\left(1+\frac{p}{p-N}\,\frac{N}{q}\right)\,\left(\frac{p}{p-N}\,\frac{N}{q}\right)^{-\frac{N}{q}\,\frac{1}{1-\frac{N}{p}+\frac{N}{q}}}.
\]
Finally, if we observe that the function
\[
a \mapsto \left(\frac{p}{p-N}\,a\right)^{-\frac{a}{1-\frac{N}{p}+a}},
\]
is bounded for $0< a\le N$, we easily get \eqref{morrey2} with a constant independent of $q$.
\end{proof}

\section{A completion space}
\label{sec:B}

In what follows, we keep on using the notation of Section \ref{sec:3}. If $\Omega\subset\mathbb{R}^N$ is an open bounded connected set, with $C^{1,\alpha}$ boundary, for some $0<\alpha<1$, for every $\delta\ll 1$ we set 
\[
\Omega_\delta=\{x\in\Omega\, :\, \mathrm{dist}(x,\partial\Omega)\le \delta\}.
\]
Let $p>2$, we still denote $u\in W^{1,p}_0(\Omega)$ the unique positive extremal of 
\[
\lambda_p(\Omega)=\min_{\varphi\in W^{1,p}_0(\Omega)}\left\{\int_\Omega |\nabla \varphi|^p\, dx\, :\, \int_\Omega |\varphi|^p=1\right\},
\]
and recall that by Lemma \ref{lm:embeddo}
\[
X^{1,2}(\Omega;|\nabla u|^{p-2}):=\left\{\varphi\in W^{1,1}(\Omega)\cap L^2(\Omega)\, :\, \int_\Omega |\nabla u|^{p-2}\,|\nabla \varphi|^2\,dx<+\infty\right\}.
\]
By $X^{1,2}_0(\Omega;|\nabla u|^{p-2})$ we still indicate the completion of $C^\infty_0(\Omega)$ with respect to the norm
\[
\|\varphi\|_{X^{1,2}(\Omega;|\nabla u|^{p-2})}=\|\varphi\|_{L^2(\Omega)}+\left(\int_\Omega |\nabla u|^{p-2}\,|\nabla \varphi|^2\,dx\right)^\frac{1}{2}.
\]
\begin{lm}
\label{lm:bonucci}
Let $\Omega\subset\mathbb{R}^N$ be an open bounded connected set, with $C^{1,\alpha}$ boundary, for some $0<\alpha<1$. 
There exists $\delta>0$ such that we have the continuous inclusion
\[
X^{1,2}(\Omega;|\nabla u|^{p-2})\subset W^{1,2}(\Omega_\delta).
\]
Moreover, we also have 
\[
X^{1,2}(\Omega;|\nabla u|^{p-2})\cap W^{1,1}_0(\Omega)= X^{1,2}_0(\Omega;|\nabla u|^{p-2}).
\]
\end{lm}
\begin{proof}
The first statement easily follows from Theorem \ref{teo:uniformi}. Indeed, we have existence of $\delta>0$ and $\mu_0>0$ such that
\[
|\nabla u|\ge \mu_0,\qquad \mbox{ in }\Omega_\delta.
\]
This entails that 
\[
\int_{\Omega_\delta} |\nabla \varphi|^2\,dx\le \mu_0^{2-p}\,\int_\Omega |\nabla u|^{p-2}\,|\nabla \varphi|^2\,dx,\qquad \mbox{ for every }\varphi\in X^{1,2}(\Omega;|\nabla u|^{p-2}),
\]
which proves the desired inclusion.
\vskip.2cm\noindent
In order to characterize the space $X^{1,2}_0(\Omega;|\nabla u|^{p-2})$, we first observe that 
\[
X^{1,2}(\Omega;|\nabla u|^{p-2})\cap W^{1,1}_0(\Omega)\supset X^{1,2}_0(\Omega;|\nabla u|^{p-2})\cap W^{1,1}_0(\Omega)=X^{1,2}_0(\Omega;|\nabla u|^{p-2}),
\]
thanks to Lemma \ref{lm:embeddo}.
\par
To prove the reverse inclusion, we now assume that $\varphi\in X^{1,2}(\Omega;|\nabla u|^{p-2})\cap W^{1,1}_0(\Omega)$. 
We need to show that there exists a sequence $\{\varphi_n\}_{n\in\mathbb{N}}\subset C^\infty_0(\Omega)$ such that 
\[
\lim_{n\to\infty}\left[\int_\Omega |\varphi-\varphi_n|^2\,dx+\int_\Omega |\nabla u|^{p-2}\,|\nabla \varphi-\nabla \varphi_n|^2\,dx\right]=0.
\]
By \cite[Theorem 18.7]{Le}, we know that 
\[
W^{1,1}_0(\Omega)=\Big\{\psi\in W^{1,1}(\Omega)\, :\, \mathrm{Tr}_{\partial\Omega}(\psi)=0\Big\},
\]
where we denote by $\mathrm{Tr}_{\partial\Omega}$ the trace of a function.
Thus $\varphi$ has a vanishing trace, as an element of $W^{1,1}(\Omega)$. On the other hand, from the first part of the proof, we know that $\varphi\in W^{1,2}(\Omega_\delta)$ and observe that we have a well-defined trace operator $\mathrm{Tr}_{\partial\Omega}:W^{1,2}(\Omega_\delta)\to L^2(\partial\Omega)$, see again \cite[Chapter 18]{Le}. By uniqueness of the trace, we thus get that
\[
\varphi\in \Big\{\psi\in W^{1,2}(\Omega_\delta)\, :\, \mathrm{Tr}_{\partial\Omega}(\psi)=0\Big\},
\]
and the latter coincides with the closure in $W^{1,2}(\Omega_\delta)$ of functions in $C^\infty(\overline{\Omega_\delta})$ that vanish in a neighborhood of $\partial\Omega$, thanks to the usual construction of the trace operator on Lipschitz sets (it is sufficient to adapt the proof of the aforementioned \cite[Theorem 18.7]{Le}). Then there exists a sequence $\{\psi_n\}_{n\in\mathbb{N}}\subset C^\infty(\overline{\Omega_\delta})$ such that 
\[
\lim_{n\to\infty} \|\psi_n-\varphi\|_{W^{1,2}(\Omega_\delta)}=0,
\]
with each $\psi_n$ vanishing in a neighborhood of $\partial\Omega$. Moreover, by means of standard convolution techniques, we can construct a sequence $\{f_n\}_{n\in\mathbb{N}}\subset C^\infty(\Omega)$ such that 
\[
\lim_{n\to\infty}\left[\int_{\Omega\setminus\Omega_\frac{\delta}{4}} |\varphi-f_n|^2\,dx\int_{\Omega\setminus\Omega_\frac{\delta}{4}} |\nabla u|^{p-2}\,|\nabla \varphi-\nabla f_n|^2\,dx\right]=0.
\]
We now take two smooth cut-off functions $\eta_\delta$ and $\zeta_\delta$ such that
\begin{itemize}
\item $\eta_\delta\in C^\infty(\overline\Omega)$ and 
\[
\eta_\delta\equiv 1 \mbox{ on }\overline{\Omega_\frac{\delta}{2}},\quad 0\le \eta_\delta\le 1,\quad \eta_\delta\equiv 0 \mbox{ on } \Omega\setminus \Omega_{\delta}; 
\]
\item $\zeta_\delta\in C^\infty_0(\Omega)$ and
\[
\zeta_\delta\equiv 1 \mbox{ on }\Omega\setminus\Omega_\frac{\delta}{2},\quad 0\le \zeta_\delta\le 1,\quad \zeta_\delta\equiv 0 \mbox{ on } \Omega_\frac{\delta}{4}.
\]
\end{itemize}
Finally, we set
\[
\varphi_n=\zeta_\delta\,f_n+\eta_\delta\,(1-\zeta_\delta)\,\psi_n.
\]
Observe that there is no mystery in such a choice: this function has the crucial feature 
\[
\zeta_\delta+\eta_\delta\,(1-\zeta_\delta)\equiv 1,\qquad \mbox{ on }\Omega,
\]
i.e. $\zeta_\delta$ and $\eta_\delta\,(1-\zeta_\delta)$ form a (very simple) partition of unity. This in particular implies
\begin{equation}
\label{gradi}
\nabla \zeta_\delta=-\nabla(\eta_\delta\,(1-\zeta_\delta)),\qquad \mbox{ on }\Omega.
\end{equation}
By construction we have $\varphi_n\in C^\infty_0(\Omega)$ and it is easily seen that
\[
\begin{split}
\lim_{n\to\infty}\int_\Omega |\varphi_n-\varphi|^2\,dx=0.
\end{split}
\]
As for the gradients, we have 
\[
\begin{split}
\int_\Omega |\nabla u|^{p-2}\,|\nabla \varphi_n-\nabla \varphi|^2\,dx&\le 2\,\int_\Omega |\nabla u|^{p-2}\,|\nabla \zeta_\delta\,f_n+\nabla (\eta_\delta\,(1-\zeta_\delta))\, \psi_n|^2\,dx\\
&+2\,\int_\Omega|\nabla u|^{p-2}\, |\zeta_\delta\,\nabla (f_n-\varphi)+\eta_\delta\,(1-\zeta_\delta)\,\nabla (\psi_n-\varphi)|^2\,dx\\
&\le C_\delta\,\int_{\Omega_\frac{\delta}{2}\setminus\Omega_\frac{\delta}{4}}|\nabla u|^{p-2}\, |f_n-\psi_n|^2\,dx\\
&+ C\,\int_{\Omega\setminus \Omega_\frac{\delta}{4}}|\nabla u|^{p-2}\, |\nabla (f_n-\varphi)|^2\,dx+C\,\int_{\Omega_\delta}|\nabla u|^{p-2}\, |\nabla (\psi_n-\varphi)|^2\,dx
\end{split}
\]
where we used \eqref{gradi} and the properties of $\eta_\delta$ and $\zeta_\delta$. By recalling the properties of $f_n$ and $\psi_n$ and using that $|\nabla u|\in L^\infty(\Omega)$, we can obtain
\[
\lim_{n\to\infty} \int_\Omega |\nabla u|^{p-2}\,|\nabla \varphi_n-\nabla \varphi|^2\,dx=0.
\]
This concludes the proof.
\end{proof}

\section{Uniform negative integrability for the gradient}
\label{sec:C}

The goal of this section is to a provide a suitable Riesz--type estimate on some {\it negative} powers of $|\nabla u_q|$ (see Theorem \ref{teo:sweden} below), where $u_q$ denotes a positive extremal of \eqref{eq:pq}. This is the cornerstone on which Theorem \ref{teo:WSI}, Corollary \ref{coro:nakata} and Corollary \ref{lem:compact}
 are built. 
As we have seen, these in turn are crucial tools for the proof of our main result. 
 The proofs are taken directly from \cite{DS}, but as explained in the introduction we need a uniform control of the a priori estimates. Occasionally, we will use the abbreviated notation $u$ in place of $u_q$, when clear from the context.
\subsection{The linearized equation}
We first observe that $u\in C^{1,\chi}(\overline\Omega)$ and the critical set
\[
Z:=\Big\{x\in\Omega\, :\, |\nabla u|=0\Big\},
\]
is a compact set contained in $\Omega$, thanks to Theorem \ref{teo:uniformi}. Thus $\Omega\setminus Z$ is an open set. Moreover, on this set the equation \eqref{eq:eq} is not degenerate, thus by classical Elliptic Regularity we can infer that $u\in C^2(\Omega\setminus Z)$. We then take $\psi\in C^\infty_0(\Omega\setminus Z)$ and test the weak formulation of \eqref{eq:eq} against a partial derivative $\psi_{x_i}$. The regularity of $u$ on $\Omega\setminus Z$ permits to integrate by parts, so to obtain 
\begin{equation}
\label{eq:lineqweak}
\begin{split}
\int_\Omega  \langle |\nabla u|^{p-2}\,\nabla u_{x_i},\nabla \psi\rangle\,dx&+(p-2)\,\int_\Omega |\nabla u|^{p-4}\,\langle \nabla u, \nabla u_{x_i}\rangle\, \langle\nabla u,\nabla \psi \rangle\, dx\\
& =(q-1)\, \lambda_{p,q}\, \int_\Omega u^{q-2}\,u_{x_i}\,\psi \,dx, \qquad \mbox{ for every } \psi\in C^\infty_0(\Omega\setminus Z).
\end{split}
\end{equation}
Here we used the more compact notation $\lambda_{p,q}$ for $\lambda_{p,q}(\Omega)$.
By density, we can even admit test functions 
$\psi\in W^{1,2}(\Omega)$ with compact support in the open set $\Omega\setminus Z$ in \eqref{eq:lineqweak}, thanks to \cite[Lemma 9.5]{Brezis}.
\vskip.2cm\noindent
\begin{oss}[Hessian terms]
\label{oss:hessian}
We seize the opportunity to mention that, since $u\in L^\infty(\Omega)$, the right-hand side of \eqref{eq:eq} is bounded. Then we have
\begin{equation}
\label{CM}
|\nabla u|^{p-2}\,\nabla u\in W^{1,2}_{\rm loc}(\Omega), 
\end{equation}
thanks to \cite[Theorem 2.1]{CM} (see also \cite[Theorem 1.1]{ACF} and \cite[Theorem 1.2]{GM} for some generalizations). In addition, as noted in \cite[Remark 2.3]{DS}, the weak gradient of $|\nabla u|^{p-2}\,\nabla u$ coincides with the classical gradient in $\Omega\setminus Z$ (where $u$ is in fact smooth), while  
\[
\nabla (|\nabla u|^{p-2}\,\nabla u)=0,\qquad \mbox{ a.\,e. on } Z,
\]
since by definition $Z$ coincides with the zero level set of $|\nabla u|^{p-2}\,\nabla u$, thus it is sufficient to use a standard property of Sobolev functions (see for example \cite[Theorem 6.19]{LL}). By further observing that $|Z|=0$ (by virtue of \cite[Theorem 2.3]{DS} or directly from \eqref{DS} below), we can actually say that the weak gradient of $|\nabla u|^{p-2}\,\nabla u$ coincides with the classical gradient almost everywhere in $\Omega$. In light of this discussion, we will keep on writing the formula
\[
(|\nabla u|^{p-2}\,\nabla u)_{x_i}=|\nabla u|^{p-2}\,\nabla u_{x_i}+(p-2)\,\langle \nabla u,\nabla u_{x_i}\rangle\,\nabla u,
\]
and observe that this expression makes sense almost everywhere on $\Omega$. 
\par
The same remarks apply whenever we deal with functions of the form $f(|\nabla u|^{p-1})$, where $f$ is a locally Lipschitz function. 
Indeed, since $u$ is globally Lipschitz by Theorem \ref{teo:uniformi}, any function of the form $f(|\nabla u|^{p-1})$ with $f$ locally Lipschitz lies in $W^{1,2}_\text{loc}(\Omega)$, thanks to \eqref{CM}. This justifies the admissibility of test functions used in the proof of Proposition \ref{prop:singest} below, as well as the use of the relevant Chain Rule formula.
\par
In the same spirit, we will use the notation $D^2 u$ for the matrix function defined by
\[
\left\{\begin{array}{rl} 0, & \mbox{ on } Z,\\
D^2 u, & \mbox{ elsewhere},
\end{array}
\right.
\]
and note that this coincides almost everywhere with the classical Hessian.
\end{oss}
\subsection{Intermediate estimates}
In what follows we will only treat the case $N\ge3$ in detail. The case $N=2$ can be treated with very minor modifications. The proposition below together with Corollary \ref{cor:dg2} correspond to \cite[Theorem 2.2]{DS}.
\begin{prop}
\label{prop:firstsingest} 
Let $2<p<q_0<p^*$ and let $\Omega\subset\mathbb{R}^N$ be an open bounded connected set, with boundary of class $C^{1,\alpha}$, for some $0<\alpha<1$. For every $p\le q\le q_0$, let $u\in W^{1,p}_0(\Omega)$ be a positive minimizer of \eqref{eq:pq}. Let $\beta\in [0,1)$ and 
\[
\left\{\begin{array}{cc}
\gamma<N-2,& \mbox{ if }N\ge 3,\\
\gamma\leq 0,& \mbox{ if } N=2.
\end{array}
\right.
\] 
Then for every $i\in\{1,\dots,N\}$ and every open set $E\Subset \Omega$ we have
\[
|u_{x_i}|^\frac{p-2-\beta}{2}\,u_{x_i}\in W^{1,2}(E).
\]
Moreover, if we set $Z_i=\{y\in \Omega\,:\, u_{x_i}(y)=0\}$, we have the estimate
\begin{equation}
\label{erik1}
\sup_{x\in\Omega}\int_{E\setminus Z_i} \frac{|\nabla u(y)|^{p-2}\,|u_{x_i}(y)|^{-\beta}\, |\nabla u_{x_i}(y)|^2}{|x-y|^\gamma}\,dy\leq C_1,
\end{equation}
for some $C_1=C_1(N,p,q_0,\alpha,\beta,\gamma,\Omega,\mathrm{dist}(E,\partial \Omega))>0$. 
\end{prop}
\begin{proof} Without loss of generality, we prove the result for $\gamma\geq 0$. The heuristic idea is to test the linearized equation \eqref{eq:lineqweak} with $u_{x_i}\,|u_{x_i}|^{-\beta} \,|x-y|^{-\gamma}\,\phi^2$, where $\phi$ is a smooth cut-off function. In order to do this rigorously, we fix $x\in \Omega$ and for $0<\varepsilon<1$ use the test function
$$
\psi(y)=G_\varepsilon(u_{x_i}(y))\,|u_{x_i}(y)|^{-\beta}\,\big(|x-y|+\varepsilon\big)^{-\gamma}\,\phi(y)^2,
$$
where $\phi\in C_0^\infty(\Omega)$ is such that 
$$
\phi\equiv 1 \mbox{ on }E,\qquad 0\le \phi\le 1,\qquad\|\nabla \phi\|_{L^\infty(\Omega)}\le \frac{C}{\mathrm{dist}(E,\partial \Omega)},
$$
and the odd Lipschitz function $G_\varepsilon$ is given by
\[
G_\varepsilon(t)=\max\Big\{t-\varepsilon,0\Big\},\ \mbox{ for }t\ge 0,\qquad \mbox{ and }\qquad G_\varepsilon(t)=-G_\varepsilon(-t),\ \mbox{ for }t\le 0.
\]
The function $\psi$ is a product of a Lipschitz function of $u_{x_i}$, which vanishes in a neighborhood of the critical set $Z$, and a smooth function with compact support in $\Omega$. 
Therefore $\psi$ is an admissible   test function for \eqref{eq:lineqweak}. This gives
\[
\begin{split}
&\int_\Omega \frac{|\nabla u|^{p-2}\,|\nabla u_{x_i}|^2\, |u_{x_i}|^{-\beta}}{\big(|x-y|+\varepsilon\big)^{\gamma}}\left[G_\varepsilon'(u_{x_i})-\beta\, \frac{G_\varepsilon(u_{x_i})}{u_{x_i}}\right]\,\phi^2\,dy \\
&+(p-2)\int_\Omega \frac{|\nabla u|^{p-4}\,\big(\langle \nabla u,\nabla u_{x_i}\rangle\big)^2\, |u_{x_i}|^{-\beta}}{\big(|x-y|+\varepsilon\big)^{\gamma}}\,\left[G_\varepsilon'(u_{x_i})-\beta \frac{G_\varepsilon(u_{x_i})}{u_{x_i}}\right]\phi^2\,dy \\
&+2\,\int_\Omega \frac{|\nabla u|^{p-2}\,\langle\nabla u_{x_i}, \nabla \phi\rangle \, |u_{x_i}|^{-\beta}}{\big(|x-y|+\varepsilon\big)^\gamma}\,\phi\,G_\varepsilon(u_{x_i})\,dy\\
& +2\,(p-2)\int_\Omega \frac{|\nabla u|^{p-4}\,\langle \nabla u,\nabla u_{x_i}\rangle\, \langle \nabla u, \nabla \phi \rangle\,|u_{x_i}|^{-\beta}}{\big(|x-y|+\varepsilon\big)^\gamma}\,\phi \,G_\varepsilon(u_{x_i})\,dy\\
&+\int_\Omega |\nabla u|^{p-2}\,\left\langle \nabla u_{x_i}, \nabla \left((|x-y|+\varepsilon)^{-\gamma}\right)\right\rangle   \,G_\varepsilon(u_{x_i})|u_{x_i}|^{-\beta}\phi^2\,dy\\
& +(p-2)\,\int_\Omega |\nabla u|^{p-4}\,\langle \nabla u, \nabla u_{x_i}\rangle\,\left \langle \nabla u,\nabla \left(\big(|x-y|+\varepsilon\big)^{-\gamma}\right)\right\rangle\, G_\varepsilon(u_{x_i})\,|u_{x_i}|^{-\beta}\,\phi^2\,dy\\
&=(q-1)\,\lambda_{p,q}\,\int_\Omega \frac{u^{q-2}\,G_\varepsilon(u_{x_i})\,|u_{x_i}|^{-\beta}}{\big(|x-y|+\varepsilon\big)^\gamma}\,\phi^2\,dy.
\end{split}
\]
Note that\footnote{Observe that the function is even, thus it is sufficient to check the inequality for $t\ge 0$.}
\[
G_\varepsilon'(t)-\beta\, \frac{G_\varepsilon(t)}{t}\geq 0,\qquad \mbox{ for every } |t|\not=\varepsilon,
\]
therefore, by dropping the second term on the left-hand side and using the Cauchy-Schwarz inequality, we get
\begin{equation}
\label{esta}
\begin{split}
\int_\Omega \frac{|\nabla u|^{p-2}|\nabla u_{x_i}|^{2}| u_{x_i}|^{-\beta}}{\big(|x-y\big|+\varepsilon\big)^\gamma}\, &\left[G_\varepsilon'(u_{x_i})-\beta \frac{G_\varepsilon(u_{x_i})}{u_{x_i}}\right]\,\phi^2\,dy\\
&\leq 2\,(p-1)\,\int_\Omega \frac{|\nabla u|^{p-2}\,|\nabla u_{x_i}|\,|u_{x_i}|^{-\beta}}{\big(|x-y|+\varepsilon\big)^\gamma}\,  |G_\varepsilon(u_{x_i})|\,\phi\,|\nabla \phi|\,dy\\
&+ \gamma\,(p-1)\,\int_\Omega \frac{|\nabla u|^{p-2}\,|\nabla u_{x_i}|\,|u_{x_i}|^{-\beta}}{\big(|x-y|+\varepsilon\big)^{\gamma+1}}\, |G_\varepsilon(u_{x_i})|\,\phi^2\,dy\\
&+(q-1)\, \lambda_{p,q}\,\int_\Omega \frac{u^{q-2}\,|G_\varepsilon(u_{x_i})|\,|u_{x_i}|^{-\beta}}{\big(|x-y|+\varepsilon\big)^\gamma}\,\phi^2\,dy=:I_1+I_2+I_3.
\end{split}
\end{equation}
By using Proposition \ref{prop:Linfty}, Theorem \ref{teo:uniformi}, the properties of the cut-off function $\phi$ and recalling  that $\lambda_{p,q}\leq \Lambda_0=\Lambda_0(N,p,q_0,\Omega)$ (see the proof of Theorem \ref{teo:uniformi}), the third term can be estimated as
\[
\begin{split}
I_3\leq C\, \int_\Omega\frac{1}{|x-y|^\gamma}\,dy,
\end{split}
\]
for a constant depending $C$ on $N,p,q_0,\alpha,\beta$ and $\Omega$, only. In turn, the last integral is easily estimated as follows
\begin{equation}
\label{kernel}
\begin{split}
\int_\Omega\frac{1}{|x-y|^\gamma}\,dy&\le \int_{\{y\in\mathbb{R}^N\, :\, |y-x|\le \mathrm{diam}(\Omega)\}} \frac{1}{|y-x|^\gamma}\,dy\\
&=N\,\omega_N\, \int_0^{\mathrm{diam}(\Omega)} \varrho^{N-1-\gamma}\,d\varrho=\frac{N\,\omega_N}{(N-\gamma)}\, \Big(\mathrm{diam}(\Omega)\Big)^{N-\gamma}.
\end{split}
\end{equation}
In the last integral we used that $\gamma<N$, thanks to the stronger assumption $\gamma<N-2$. 
\par
As for the terms $I_1$ and $I_2$, we first observe that by using 
\begin{equation}
\label{kernelg}
(|x-y|+\varepsilon)^\gamma\ge \frac{(|x-y|+\varepsilon)^{\gamma+1}}{\mathrm{diam}(\Omega)+1}, \qquad \mbox{ for every } x,y\in\Omega,
\end{equation}
we have
\[
I_1+I_2\le C\,\int_\Omega \frac{|\nabla u|^{p-2}\,|\nabla u_{x_i}|\,|u_{x_i}|^{-\beta}}{\big(|x-y|+\varepsilon\big)^{\gamma+1}}\, |G_\varepsilon(u_{x_i})|\,\phi\,[\phi+|\nabla \phi|]\,dy,
\]
for some $C=C(p,\gamma,\mathrm{diam}(\Omega))>0$. Thus, from \eqref{esta}, we have obtained
\begin{equation}
\label{esta2}
\begin{split}
\int_\Omega \frac{|\nabla u|^{p-2}\,|\nabla u_{x_i}|^{2}\,| u_{x_i}|^{-\beta}}{\big(|x-y\big|+\varepsilon\big)^\gamma}\, &\left[G_\varepsilon'(u_{x_i})-\beta \,\frac{G_\varepsilon(u_{x_i})}{u_{x_i}}\right]\,\phi^2\,dy\\
&\le C+C\,\int_\Omega \frac{|\nabla u|^{p-2}\,|\nabla u_{x_i}|\,|u_{x_i}|^{-\beta}}{\big(|x-y|+\varepsilon\big)^{\gamma+1}}\, |G_\varepsilon(u_{x_i})|\,\phi\,[\phi+|\nabla \phi|]\,dy.
\end{split}
\end{equation}
In order to estimate the last integral, we use Young's inequality. For every $\delta>0$, we have 
\begin{equation}
\label{estb}
\begin{split}
C\,\int_\Omega \frac{|\nabla u|^{p-2}\,|\nabla u_{x_i}|\,|u_{x_i}|^{-\beta}}{\big(|x-y|+\varepsilon\big)^{\gamma+1}}\, &|G_\varepsilon(u_{x_i})|\,\phi\,[\phi+|\nabla \phi|]\,dy\\
&\leq \frac{\delta}{2}\, \int_\Omega \frac{|\nabla u|^{{p-2}}\,|\nabla u_{x_i}|^2\,|u_{x_i}|^{-\beta}}{\big(|x-y|+\varepsilon\big)^\gamma}\,\frac{|G_\varepsilon(u_{x_i})|}{|u_{x_i}|}\, \phi^2\,dy\\
&+ \frac{C^2}{2\,\delta}\, \int_\Omega \frac{|\nabla u|^{p-2}\,|u_{x_i}|^{1-\beta}}{\big(|x-y|+\varepsilon\big)^{\gamma+2}}\,|G_\varepsilon(u_{x_i})| \, [\phi+|\nabla \phi|]^2\,dy.
\end{split}
\end{equation}
We make the choice $\delta=1-\beta$, which is feasible since $\beta<1$. Then observe that
\[
(1-\beta)\,\frac{|G_\varepsilon(t)|}{|t|}\leq  G_\varepsilon'(t)-\beta \frac{G_\varepsilon(t)}{t},\qquad \mbox{ for every } |t|\not=\varepsilon.
\]
Then the first term in the right-hand side of \eqref{estb} can now be estimated by
$$
\frac12 \int_\Omega \frac{|\nabla u|^{p-2}\,|\nabla u_{x_i}|^2\,|u_{x_i}|^{-\beta}}{\big(|x-y|+\varepsilon\big)^\gamma}\,\left[G_\varepsilon'(u_{x_i})-\beta\, \frac{G_\varepsilon(u_{x_i})}{u_{x_i}}\right]\,\phi^2\,dy,
$$
which can be absorbed into the right-hand side of \eqref{esta2}. Thus, up to now, we obtained 
\begin{equation}
\label{esta3}
\begin{split}
\int_\Omega \frac{|\nabla u|^{p-2}\,|\nabla u_{x_i}|^{2}\,| u_{x_i}|^{-\beta}}{\big(|x-y\big|+\varepsilon\big)^\gamma}\, &\left[G_\varepsilon'(u_{x_i})-\beta \frac{G_\varepsilon(u_{x_i})}{u_{x_i}}\right]\,\phi^2\,dy\\
&\le C+C\,\int_\Omega \frac{|\nabla u|^{p-2}\,|u_{x_i}|^{1-\beta}}{\big(|x-y|+\varepsilon\big)^{\gamma+2}}\,|G_\varepsilon(u_{x_i})| \, [\phi+|\nabla \phi|]^2\,dy,
\end{split}
\end{equation}
possibly for a different constant $C$, independent of $x\in\Omega,\varepsilon\in(0,1)$ and $q\in[p,q_0]$.
Using that $|G_\varepsilon(t)|\leq |t|$, together with the Lipschitz bound of Theorem \ref{teo:uniformi} and the properties of $\phi$, the last integral of \eqref{esta3} can be estimated by
\[
\int_\Omega\frac{|\nabla u|^{p-2}\,|u_{x_i}|^{2-\beta}}{\big(|x-y|+\varepsilon\big)^{\gamma+2}}\, [\phi+|\nabla \phi|]^2\,dy\leq C\,\int_\Omega \frac{1}{|x-y|^{\gamma+2}}\,dy,
\]
and the last integral is uniformly (in $x\in\Omega$) bounded by a constant depending only on $N$ and $\mathrm{diam}(\Omega)$, exactly as in \eqref{kernel} (here we crucially use that $\gamma<N-2$).
\par
From \eqref{esta3} and using that $\phi\equiv 1$ on $E$, we thus finally obtain
\begin{equation}
\label{erikfinal}
\int_E \frac{|\nabla u|^{p-2}\,| u_{x_i}|^{-\beta}\,|\nabla u_{x_i}|^{2}}{\big(|x-y\big|+\varepsilon\big)^\gamma}\,\left[G_\varepsilon'(u_{x_i})-\beta \frac{G_\varepsilon(u_{x_i})}{u_{x_i}}\right]\,dy\le C,
\end{equation}
for some $C=C(N,p,q_0,\gamma,\alpha,\beta,\Omega,\mathrm{dist}(E,\partial\Omega))>0$. If we introduce the function
\[
F_\varepsilon(t)=\int_0^t |\tau|^\frac{p-2-\beta}{2}\,\sqrt{G_\varepsilon'(\tau)-\beta \frac{G_\varepsilon(\tau)}{\tau}}\,d\tau,
\]
from \eqref{erikfinal} we can infer in particular that
\[
\int_E |\nabla F_\varepsilon(u_{x_i})|^2\,dx\le C.
\]
Observe that we also used that $|\nabla u|^{p-2}\ge |u_{x_i}|^{p-2}$.
Moreover, by construction of $F_\varepsilon$, it is not difficult to see that 
\[
\int_E |F_\varepsilon(u_{x_i})|^2\,dx\le C_\beta\, \int_E |u_{x_i}|^{p-\beta}.
\]
This shows that, up to extracting an infinitesimal sequence $\{\varepsilon_n\}_{n\in\mathbb{N}}\subset (0,1)$, we have that $F_\varepsilon(u_{x_i})$ converges weakly in $W^{1,2}(E)$ to a function $F\in W^{1,2}(E)$. By using that 
\begin{equation}
\label{limitG}
\lim_{\varepsilon\to 0} \left[G_\varepsilon'(\tau)-\beta \frac{G_\varepsilon(\tau)}{\tau}\right]=1-\beta,\qquad \mbox{ for } \tau\not =0,
\end{equation}
we get that
\[
\lim_{\varepsilon\to 0}F_\varepsilon(t)=(1-\beta)\,\int_0^t |\tau|^\frac{p-2-\beta}{2}\,d\tau=\frac{2\,(1-\beta)}{p-\beta}\,|t|^\frac{p-2-\beta}{2}\,t,
\]
and thus
\[
\lim_{\varepsilon\to 0} F_\varepsilon(u_{x_i})=\frac{2\,(1-\beta)}{p-\beta}\,|u_{x_i}|^\frac{p-2-\beta}{2}\,u_{x_i},\qquad \mbox{ a.\,e. in }E.
\]
This permits to identify the limit function $F$, thanks to \cite[Lemme 4.8]{ka}, which is then given by
\[
\frac{2\,(1-\beta)}{p-\beta}\,|u_{x_i}|^\frac{p-2-\beta}{2}\,u_{x_i}=F\in W^{1,2}(E).
\]  
This already shows the first part of the statement. 
\par
Finally, by using Fatou's Lemma together with \eqref{limitG}, 
we may take the limit as $\varepsilon$ goes to $0$ in \eqref{erikfinal} and obtain 
$$
\int_{E\setminus Z_i} \frac{|\nabla u|^{p-2}\,| u_{x_i}|^{-\beta}\,|\nabla u_{x_i}|^2}{|x-y|^\gamma}\,dy\leq C.
$$
The previous bound holds uniformly with respect to $x\in\Omega$, thus we get the desired conclusion.
\end{proof}
We then have the following immediate consequence of Proposition \ref{prop:firstsingest}.
\begin{coro}
\label{cor:dg2}
Under the assumptions of Proposition \ref{prop:firstsingest}, for every $\beta\in(-\infty,1)$ we have 
\[
\sup_{x\in\Omega}\int_{E} \frac{|\nabla u(y)|^{p-2-\beta }\, |D^2 u(y)|^2}{|x-y|^\gamma}\,dy\leq  C_2,
\]
for some $C_2=C_2(N,p,q_0,\alpha,\beta,\gamma,\Omega,\mathrm{dist}(E,\partial \Omega))>0$.
\end{coro}
\begin{proof}
We first suppose that $0\le \beta<1$. From \eqref{erik1}, by using that $|u_{x_i}|^{-\beta}\ge |\nabla u|^{-\beta}$, we immediately get
\[
\sup_{x\in\Omega}\int_{E\setminus Z_i} \frac{|\nabla u(y)|^{p-2-\beta}\,|\nabla u_{x_i}(y)|^2}{|x-y|^\gamma}\,dy\leq C_1, \qquad \mbox{ for } i=1,\dots,N.
\]
We then observe that $u_{x_i}\in C^1(\Omega\setminus Z)$ and thus it belongs to $W^{1,1}_{\rm loc}(\Omega\setminus Z)$. By appealing to \cite[Theorem 6.19]{LL}, we have 
\[
\nabla u_{x_i}=0,\qquad \mbox{ a.\,e. in } Z_i\cap (\Omega\setminus Z).
\]
By using this fact and summing over $i=1,\dots,N$, we get the claimed inequality, by recalling that $|Z|=0$, see Remark \ref{oss:hessian}.
\par
The case $\beta< 0$ can now be reduced to the case $\beta=0$: it is sufficient to use that $\|\nabla u\|_{L^\infty(\Omega)}\le L<+\infty$ by Theorem \ref{teo:uniformi}, thus we get 
\[
\sup_{x\in\Omega}\int_{E} \frac{|\nabla u(y)|^{p-2-\beta }\, |D^2 u(y)|^2}{|x-y|^\gamma}\,dy\le L^{-\beta}\,\sup_{x\in\Omega}\int_{E} \frac{|\nabla u(y)|^{p-2}\, |D^2 u(y)|^2}{|x-y|^\gamma}\,dy.
\]
This concludes the proof.
\end{proof}
We also need the following a priori estimate, which corresponds to \cite[Theorem 2.3]{DS}.
\begin{prop} 
\label{prop:singest} 
Under the assumptions of Proposition \ref{prop:firstsingest}, for every $K\Subset E\Subset \Omega$ and every $b<1$ we have
\[
\sup_{x\in\Omega}\int_{K} \frac{1}{|\nabla u(y)|^{(p-1)\,b}|x-y|^\gamma}\,dy\leq C_3\,\left(\inf_E u\right)^{1-q}\,\left(1+\left(\inf_E u\right)^{1-q}\right),
\]
where $C_3=C_3(N,p,q_0,\alpha,b,\gamma,\Omega,\mathrm{dist}(K,\partial E),\mathrm{dist}(E,\partial\Omega))>0$.
\end{prop}
\begin{proof} Without loss of generality, we prove the result for $\gamma\geq 0$. The heuristic idea is to test equation \eqref{eq:eq} with $|\nabla u|^{-(p-1)\,b}\,|x-y|^{-\gamma}\phi$, where $\phi\in C_0^\infty(E)$ is a cut-off function such that 
\[
\phi \equiv 1 \mbox { in } K,\qquad 0\le \phi\le 1,\qquad \|\nabla \phi\|_{L^\infty(E)}\le \frac{C}{\mathrm{dist}(K,\partial E)}.
\]
To make this precise, we fix $x\in \Omega$ and use for every $\varepsilon>0$ the test function 
\begin{equation}
\label{testerik}
\psi(y)= (|\nabla u(y)|^{p-1}+\varepsilon)^{-b}\,\big(|x-y|+\varepsilon\big)^{-\gamma}\,\phi(y).
\end{equation}
This is a product of a Lipschitz function of $|\nabla u|^{p-1}$ and a smooth function with compact support in $\Omega$. In light of \eqref{CM}, we have that this function belongs to $W^{1,2}_0(\Omega)$. If we now use that $u\in L^\infty(\Omega)$ and $\nabla u\in L^\infty(\Omega)$, we see that in the weak formulation of \eqref{eq:eq}
we can in particular admit test functions $\psi\in W^{1,2}_0(\Omega)$.
Therefore the test function in \eqref{testerik} is feasible. 
\par
This gives
\[
\begin{split}
\lambda_{p,q}\,\int_\Omega \frac{u^{q-1}}{(|\nabla u|^{p-1}+\varepsilon)^b}&\,\frac{\phi}{\big(|x-y|+\varepsilon\big)^\gamma}\,  dy\\
&=\int_\Omega \frac{\langle|\nabla u|^{p-2}\,\nabla u, \nabla \phi \rangle}{(|\nabla u|^{p-1}+\varepsilon)^b}\,\frac{1}{\big(|x-y|+\varepsilon\big)^\gamma}\,dy\\
&-b\,(p-1)\,\int_\Omega \frac{|\nabla u|^{2\,p-5}}{(|\nabla u|^{p-1}+\varepsilon)^{b+1}}\,\langle D^2 u\,\nabla u,\nabla u \rangle\, \frac{\phi}{\big(|x-y|+\varepsilon\big)^\gamma}\,dy\\
&+\int_\Omega \frac{|\nabla u|^{p-2}}{(|\nabla u|^{p-1}+\varepsilon)^b}\,\langle \nabla u,\nabla \big(|x-y|+\varepsilon\big)^{-\gamma}\rangle\, \phi\, dy.
\end{split}
\]
Observe that we used Remark \ref{oss:hessian}, to compute the gradient of $(|\nabla u(y)|^{p-1}+\varepsilon)^{-b}$.
Using that 
\[
u\ge \inf_E u>0,
\]
and the lower bound on $\lambda_{p,q}$ given by \eqref{lowerlambda}, we obtain 
\begin{equation}
\label{tomte}
\begin{split}
\left(\inf_{E} u\right)^{q-1}\,\int_\Omega \frac{\phi}{(|\nabla u|^{p-1}+\varepsilon)^b}\, \frac{1}{\big(|x-y|+\varepsilon\big)^\gamma}\, dy&\leq C\,\int_\Omega \frac{|\nabla u|^{p-1}\,|\nabla \phi| }{(|\nabla u|^{p-1}+\varepsilon)^b}\, \frac{dy}{\big(|x-y|+\varepsilon\big)^\gamma}\\
&+ C\,b\,\int_\Omega \frac{|\nabla u|^{2\,p-5}\,\phi}{(|\nabla u|^{p-1}+\varepsilon)^{b+1}}\,\frac{\langle D^2 u\,\nabla u,\nabla u\rangle}{\big(|x-y|+\varepsilon\big)^\gamma}\,dy\\
&+ C\,\gamma\,\int_\Omega \frac{|\nabla u|^{p-1}\,\phi}{(|\nabla u|^{p-1}+\varepsilon)^b}\,\frac{dy}{\big(|x-y|+\varepsilon\big)^{\gamma+1}}\\
&=:J_1+J_2+J_3,
\end{split}
\end{equation}
where $C=C(N,p,q_0,\Omega)>0 $. We observe that, by using \eqref{kernelg}, we have
\[
J_1+J_3\leq C\,\int_\Omega \frac{|\nabla u|^{p-1}}{(|\nabla u|^{p-1}+\varepsilon)^b}\, \frac{\phi+|\nabla \phi|}{\big(|x-y|+\varepsilon\big)^{\gamma+1}}\,dy,
\]
for a constant $C$ depending on $N,p,q_0, \Omega$ and $\gamma$. We can then go on by observing that $|\nabla u|^{p-1}\le |\nabla u|^{p-1}+\varepsilon$ and $b<1$, using the Lipschitz estimate of Theorem \ref{teo:uniformi} and the properties of $\phi$. This gives
\[
J_1+J_3\leq C\,\int_\Omega \frac{1}{\big(|x-y|+\varepsilon\big)^{\gamma+1}}\,dy,
\]
for a constant $C=C(N,p,q_0,\alpha,b,\gamma,\Omega, \mathrm{dist}(K,\partial E))>0$.
Then we can estimate the last integral as in \eqref{kernel}.
For $J_2$ we have
\[
\begin{split}
J_2&\leq C\,\int_\Omega \frac{|\nabla u|^{2\,p-3}}{(|\nabla u|^{p-1}+\varepsilon)^{b+1}}\,\frac{|D^2 u|}{\big(|x-y|+\varepsilon\big)^\gamma}\,\phi \,dy.
\end{split}
\]
In the last integral above, we use Young's inequality as follows 
\[
\begin{split}
C\,\int_\Omega \frac{|\nabla u|^{2\,p-3}}{(|\nabla u|^{p-1}+\varepsilon)^{b+1}}&\,\frac{|D^2 u|}{\big(|x-y|+\varepsilon\big)^\gamma}\,\phi \,dy \\
&\leq \frac{\left(\inf\limits_E u\right)^{q-1}}{2}\,\int_\Omega \frac{1}{(|\nabla u|^{p-1}+\varepsilon)^b}\, \frac{\phi}{\big(|x-y|+\varepsilon\big)^\gamma}\,dy\\
&  +\frac{C^2}{2}\,\left(\inf\limits_E u\right)^{1-q}\int_\Omega \frac{|\nabla u|^{4\,p-6}}{(|\nabla u|^{p-1}+\varepsilon)^{b+2}}\,\frac{|D^2 u|^2\,\phi}{\big(|x-y|+\varepsilon\big)^\gamma}\,dy\\
&\leq \frac{(\inf_E u)^{q-1}}{2}\,\int_\Omega \frac{1}{(|\nabla u|^{p-1}+\varepsilon)^b}\, \frac{\phi}{\big(|x-y|+\varepsilon\big)^\gamma}\,dy\\
&  +\frac{C^2}{2}\,\left(\inf\limits_E u\right)^{1-q}\,\int_\Omega |\nabla u|^{(2-b)\,(p-1)-2}\,\frac{|D^2 u|^2\,\phi}{\big(|x-y|+\varepsilon\big)^\gamma}\,dy.
\end{split}
\]
Note that we have 
\[
(2-b)\,(p-1)-2>p-3,
\]
and thus the last integral is finite and uniformly bounded, thanks to Corollary \ref{cor:dg2}: it is sufficient to choose in the latter
\[
\beta=p-(2-b)\,(p-1),
\]
which is feasible. We then obtain from \eqref{tomte}
\[
\begin{split}
\left(\inf_{E} u\right)^{q-1}\,\int_\Omega \frac{\phi}{(|\nabla u|^{p-1}+\varepsilon)^b}\, \frac{1}{\big(|x-y|+\varepsilon\big)^\gamma}\, dy&\leq C\,\left(1+\left(\inf\limits_E u\right)^{1-q}\right)\\
&+\frac{\left(\inf\limits_E u\right)^{q-1}}{2}\,\int_\Omega \frac{1}{(|\nabla u|^{p-1}+\varepsilon)^b}\, \frac{\phi}{\big(|x-y|+\varepsilon\big)^\gamma}\,dy,
\end{split}
\]
upon renaming the constant $C=C(N,p,q_0,\alpha,b,\gamma,\Omega,\mathrm{dist}(E,\partial\Omega))>0$.
The term on the right-hand side can now be absorbed in the left-hand side.  
Since $\phi=1$ on $K$, this implies the desired result upon letting $\varepsilon$ go to $0$ and using Fatou's Lemma.
\end{proof}
\subsection{Negative integrability} We are finally ready for the main result of this section. Again, we will only treat the case $N\ge3$ in detail.
\begin{teo}
\label{teo:sweden}
Let $2<p<q_0<p^*$ and let $\Omega\subset\mathbb{R}^N$ be an open bounded connected set, with boundary of class $C^{1,\alpha}$, for some $0<\alpha<1$. For every $p\le q\le q_0$, let $u_q\in W^{1,p}_0(\Omega)$ be a positive minimizer of \eqref{eq:pq}. Then for 
\[
\left\{\begin{array}{cc}
\gamma<N-2,& \mbox{ if }N\ge 3,\\
\gamma\leq 0,& \mbox{ if } N=2,
\end{array}
\right.
\] 
and every $r<p-1$, there exists $\mathcal{S}=\mathcal{S}(\alpha,N,p,q_0,\Omega,r,\gamma)>0$ such that
$$
\sup_{x\in\Omega}\int_\Omega \frac{1}{|\nabla u_q(y)|^r\,|y-x|^\gamma} \,dy \le \mathcal{S}.
$$
In particular, we also have
\begin{equation}
\label{DS}
\int_\Omega \frac{1}{|\nabla u_q(y)|^r} \,dy \le \widetilde{\mathcal{S}},
\end{equation}
for some $\widetilde{\mathcal{S}}=\widetilde{\mathcal{S}}(\alpha,N,p,q_0,\Omega,r)>0$.
\end{teo}
\begin{proof}
This is \cite[Theorem 2.3]{DS}: as claimed, we just want to pay attention to the dependence of the constant $\mathcal{S}$ on the data. 
We start by observing that by Theorem \ref{teo:uniformi} and \eqref{kernel}, we have
\begin{equation}
\label{minusgamma}
\begin{split}
\int_{\Omega_\delta}\frac{1}{|\nabla u_q(y)|^r\,|y-x|^\gamma} \,dy&\le \frac{1}{\mu_0^r}\,\int_{\Omega_\delta} \frac{1}{|y-x|^\gamma}\,dy\le \frac{N\,\omega_N}{(N-\gamma)\,\mu_0^r}\, \Big(\mathrm{diam}(\Omega)\Big)^{N-\gamma}.
\end{split}
\end{equation}
Here $\delta$ and $\mu_0$ are as in the statement of Theorem \ref{teo:uniformi}.
Thus we have a uniform estimate, at least when integrating in a fixed neighborhood of the boundary.
\par
In order to prove a uniform estimate on 
$$
\int_{\Omega\setminus \Omega_\delta} \frac{1}{|\nabla u_q(y)|^r\,|y-x|^\gamma} \,dy ,\qquad \mbox{ for every } x\in\Omega,
$$
we apply Proposition \ref{prop:singest} with 
\[
E=\Omega\setminus \Omega_{\delta/2},\qquad  K=\Omega\setminus \Omega_{\delta},\qquad b=\frac{r}{p-1},
\] 
and with the constant $\mu_1$ provided by Theorem \ref{teo:uniformi}. This yields
$$
\int_{\Omega\setminus \Omega_\delta} \frac{1}{|\nabla u_q(y)|^r\,|y-x|^\gamma} \,dy\leq C_2\,\mu_1^{1-q}\left(1+\mu_1^{1-q}\right).
$$
By using that 
\[
\mu_1^{1-q}\le \max\left\{\mu_1^{1-p},\mu_1^{1-q_0}\right\},\qquad \mbox{ for } p\le q\le q_0,
\]
we get the desired estimate.
\par 
Finally, the estimate \eqref{DS} is an easy consequence of the previous one, it is sufficient to take $\gamma=0$.
\end{proof}


\begin{thebibliography}{100}

\bibitem{AY} Adimurthi, S. L. Yadava, An elementary proof of the uniqueness of positive radial solutions of a quasilinear Dirichlet problem, Arch. Rational Mech. Anal., {\bf 127} (1994), 219-229.

\bibitem{AH} W. Allegretto, Y. X. Huang, A Picone's identity for the $p-$Laplacian and applications, Nonlinear Anal. {\bf 32} (1998), 819--830. 

\bibitem{AFI} G. Anello, F. Faraci, A. Iannizzotto, On a problem of Huang concerning best constants in Sobolev embeddings, Ann. Mat. Pura Appl. (4), {\bf 194} (2015), 767--779.


\bibitem{ACF} C. A. Antonini, G. Ciraolo, A. Farina, Interior regularity results for inhomogeneous anisotropic quasilinear equations, preprint (2021), available at {\tt https://arxiv.org/abs/2112.09087}

\bibitem{BF} L. Brasco, G. Franzina, An overview on constrained critical points of Dirichlet integrals, Rend. Semin. Mat. Univ. Politec. Torino, {\bf 78} (2020), 7--50.

\bibitem{Brezis} H. Brezis, {\it Functional Analysis, Sobolev Spaces and Partial Differential Equations.}  Universitext, Springer, New York (2011).

\bibitem{BO} H. Brezis, L. Oswald, Remarks on sublinear elliptic equations, Nonlinear Anal., {\bf 10} (1986), 55--64.

\bibitem{CM} A. Cianchi, V. Maz'ya, Second-order two-sided estimates in nonlinear elliptic problems,
Arch. Ration. Mech. Anal., {\bf 229} (2018), 569--599.


\bibitem{CES} D. Castorina, P. Esposito, B. Sciunzi, Degenerate elliptic equations with singular nonlinearities. Calc. Var. Partial Differential Equations, {\bf 34} (2009), 279--306.


\bibitem{DGP} L. Damascelli, M. Grossi, F. Pacella, Qualitative properties of positive solutions of semilinear elliptic equations in symmetric domains via the maximum principle, Ann. Inst. H. Poincar\'e Anal. Non Lin\'eaire, {\bf 16} (1999), 631--652.

\bibitem{DS} L. Damascelli, B. Sciunzi, Regularity, monotonicity and symmetry of positive solutions of $m-$Laplace equations, J. Differential Equations, {\bf 206} (2004), 483--515.

\bibitem{Da1} E. N. Dancer, On the influence of domain shape on the existence of large solutions of some superlinear problems, Math. Ann., {\bf 285} (1989), 647--669. 

\bibitem{Da2} E. N. Dancer, The effect of domain shape on the number of positive solutions of certain nonlinear equations, J. Diff. Eq., {\bf 74} (1988), 120--156. 

\bibitem{DiSa}  J. I. D\'iaz, J. E. Sa\'a, Existence et unicit\'e de solutions positives pour certaines \'equations elliptiques quasilin\'eaires, C. R. Acad. Sci. Paris S\'er. I Math., {\bf 305} (1987), 521--524.

\bibitem{Er2} G. Ercole, Regularity results for the best-Sobolev-constant function, Ann. Mat. Pura Appl. (4), {\bf 194} (2015), 1381--1392.

\bibitem{Er2014} G. Ercole, Sign-definiteness of $q-$eigenfunctions for a super-linear $p-$Laplacian eigenvalue problem, Arch. Math., {\bf 103} (2014), 189--194. 

\bibitem{Er} G. Ercole, Absolute continuity of the best Sobolev constant of a bounded domain, J. Math. Anal. Appl., {\bf 404} (2013), 420--428.

\bibitem{GT} D. Gilbarg, N. Trudinger, {\it Elliptic partial differential equations of second order.
Second edition.} Grundlehren der Mathematischen Wissenschaften [Fundamental Principles of Mathematical Sciences], {\bf 224}. Springer-Verlag, Berlin, 1983. 

\bibitem{GM} U. Guarnotta, S. Mosconi, A general notion of uniform ellipticity and the regularity of the stress field for elliptic equations in divergence form, to appear on Anal. PDE (2022), available at {\tt https://arxiv.org/abs/2105.12546v3}

\bibitem{Kawohl} B. Kawohl, Symmetry results for functions yielding best constants in Sobolev-type inequalities, Discrete Contin. Dynam. Systems, {\bf 6} (2000), 683--690.

\bibitem{ka}
O.\ Kavian, Introduction \`a la th\'eorie des points critiques et applications aux probl\`emes elliptiques. (French) [Introduction to critical point theory and applications to elliptic problems] Mathématiques \& Applications (Berlin) [Mathematics \& Applications], {\bf 13}. Springer-Verlag, Paris, 1993.

\bibitem{KLP} B. Kawohl, M. Lucia, S. Prashanth, Simplicity of the principal eigenvalue for indefinite quasilinear problems, 
Adv. Differential Equations, {\bf 12} (2007), 407--434.


\bibitem{KL} M. K. Kwong, Y. Li, Uniqueness of radial solutions of semilinear elliptic equations, Trans. Amer. Math. Soc., {\bf 333} (1992), 339--363. 

\bibitem{He} A. Henrot, {\it Extremum problems for eigenvalues of elliptic operators}. Frontiers in Mathematics. Birkhauser Verlag, Basel, 2006.

\bibitem{HL} R. Hynd, E. Lindgren, Extremal functions for Morrey's inequality in convex domains, Math. Ann., {\bf 375} (2019),  1721--1743.

\bibitem{IO} T. Idogawa, M. Otani, The first eigenvalues of some abstract elliptic operators,  Funkcial. Ekvac., {\bf 38} (1995), 1--9.

\bibitem{Le}  G. Leoni, {\it A first course in Sobolev spaces. Second edition.} Graduate Studies in Mathematics, {\bf 181}. American Mathematical Society, Providence, RI, 2017

\bibitem{LL} E. H. Lieb, M. Loss, {\it Analysis. Second Edition.} Graduate Studies in Mathematics, {\bf 14}. American Mathematical Society, Providence, RI, 2001.

\bibitem{Lie88} G. Lieberman, Boundary regularity for solutions of degenerate elliptic equations, 
Nonlinear Anal., {\bf 12} (1988), 1203--1219.


\bibitem{Lin} C.-S. Lin, Uniqueness of least energy solutions to a semilinear elliptic equation in $\mathbb{R}^
2$, Manuscripta Math., {\bf 84} (1994), 13--19.

\bibitem{Lindqvist} P. Lindqvist, On the equation $\operatorname{div}(|\nabla u|^{p-2}\nabla u)+\lambda |u|^{p-2}u = 0$, Proc. Amer. Math. Soc., {\bf 109}
(1990), 157--164.

\bibitem{Maz} V. Maz'ya, {\it Sobolev spaces with applications to elliptic partial differential equations}. Second, revised and augmented edition. Grundlehren der Mathematischen Wissenschaften [Fundamental Principles of Mathematical Sciences], {\bf 342}. Springer, Heidelberg, 2011. 

\bibitem{MS} H. Mikaelyan, H. Shahgholian, Hopf's lemma for a class of singular/degenerate PDE-s,  Ann. Acad. Sci. Fenn. Math., {\bf 40} (2015), 475--484.

\bibitem{Naz} A. I. Nazarov, The one-dimensional character of an extremum point of the Friedrichs inequality in
spherical and plane layers, J. Math. Sci., {\bf 102} (2000), 4473--4486.

\bibitem{Ta} 
P. Tak\'a\v{c}, On the Fredholm Alternative for the $p-$Laplacian at the First Eigenvalue, Indiana Univ. Math. J., {\bf 51} (2002), 187-237.

\end{thebibliography}
\end{document}